\documentclass[12pt,reqno]{amsart}
\usepackage{amsmath,amsfonts,amsthm,amssymb,color}
\usepackage[T1]{fontenc}
\usepackage{enumerate}
\usepackage{hyperref}
\hypersetup{
     colorlinks   = true,
     citecolor    = blue,
     linkcolor    = blue
}
\usepackage[left=1in, right=1in, top=1.1in,bottom=1.1in]{geometry}
\usepackage{verbatim}
\usepackage[active]{srcltx}
\include{srctex}
\usepackage{pdfsync}

\newcommand{\dom}{\mbox{Dom}}

\newcommand{\hq}{\hat{Q}}
\newcommand{\htt}{\hat{T}}

\newcommand{\iot}{\int_{0}^{t}}

\newcommand{\ot}{[0,t]}
\newcommand{\ott}{[0,T]}

\newcommand{\1}{{\bf 1}}

\def\RR{\mathbb{R}}

\def\EE{\mathbb{E}}

\newcommand{\C}{\mathbb C}

\newcommand{\R}{\mathbb R}
\newcommand{\N}{\mathbb N}

\newcommand{\PP}{\mathbb P}
\newcommand{\be}{\mathbb E}
\newcommand{\bp}{\mathbb{P}}

\newcommand{\ca}{\mathcal A}

\newcommand{\ce}{\mathcal E}
\newcommand{\cf}{\mathcal F}

\newcommand{\ch}{\mathcal H}

\newcommand{\ck}{\mathcal K}

\newcommand{\cn}{\mathcal N}

\newcommand{\cs}{\mathcal S}

\newcommand{\al}{\alpha}

\newcommand{\ga}{\gamma}
\newcommand{\gga}{\Gamma}
\newcommand{\ka}{\kappa}
\newcommand{\la}{\lambda}
\newcommand{\laa}{\Lambda}

\newcommand{\oom}{\Omega}

\newcommand{\si}{\sigma}
\newcommand{\te}{\theta}

\newcommand{\vp}{\varphi}

\newcommand{\ep}{\varepsilon}

\newcommand{\lp}{\left(}
\newcommand{\rp}{\right)}
\newcommand{\lc}{\left[}
\newcommand{\rc}{\right]}
\newcommand{\lcl}{\left\{}
\newcommand{\rcl}{\right\}}
\newcommand{\lln}{\left|}
\newcommand{\rrn}{\right|}
\newcommand{\lla}{\left\langle}
\newcommand{\rra}{\right\rangle}

\allowdisplaybreaks[4]
\newtheorem{theorem}{Theorem}[section]

\newtheorem{corollary}[theorem]{Corollary}

\newtheorem{definition}[theorem]{Definition}

\newtheorem{lemma}[theorem]{Lemma}

\newtheorem{proposition}[theorem]{Proposition}
\theoremstyle{remark}
\newtheorem{remark}[theorem]{Remark}

\let\Section=\section
\def\section{\setcounter{equation}{0}\Section}

\begin{document}

\title[Parabolic Anderson model]
{Spatial asymptotics for the parabolic Anderson model  \\
driven by a Gaussian rough  noise}

\author[X. Chen, Y. Hu,   D. Nualart, S. Tindel]{Xia Chen \and Yaozhong Hu
  \and David Nualart \and Samy Tindel}

\address{Xia Chen:  Department of Mathematics, University of Tennessee Knoxville,  TN 37996-1300, USA}
\email{xchen@math.utk.edu}
\thanks{X. Chen is partially supported by a grant 
from the Simons Foundation \#244767}

\address{Yaozhong Hu  and David Nualart: Department of Mathematics, University of Kansas,   Lawrence, KS 66045, USA.}
\email{yhu@ku.edu,   nualart@ku.edu}

\address{Samy Tindel:  Department of Mathematics, Purdue University, West Lafayette, IN 47907, USA}
\email{stindel@purdue.edu}
\thanks{Y. Hu is partially supported by a grant from the Simons Foundation \#209206}
\thanks{D. Nualart is supported by the NSF grant  DMS1512891 and the ARO grant FED0070445}
\thanks{S. Tindel is supported by the NSF grant  DMS1613163}


\date{\today}


\begin{abstract}
The aim of this paper is to establish the almost sure asymptotic behavior as the space variable becomes 
 large, for the solution to  the one spatial dimensional stochastic heat equation 
 driven by a Gaussian noise which  is white in time and which
  has the covariance structure of a fractional Brownian motion with Hurst parameter
 $H \in \left( \frac 14, \frac 12 \right)$ in the space variable.  
\end{abstract}

\maketitle

\section{Introduction}
 
This article is concerned with a linear stochastic heat equation on $\R_{+}\times \R$, formally written as 
\begin{equation}\label{eq:spde-pam}
\frac{\partial u}{\partial t}=\frac{1}{2}\frac{\partial^2 u}{\partial x^2}+ u \, \dot{W}\,, \quad t\ge 0, \quad x\in\mathbb{R}\, ,
\end{equation}
where $\dot{W}$ is a Gaussian  noise which is white in time and colored in space, and we are interested in regimes where the spatial behavior of  $\dot{W}$ is rougher than white noise. More specifically, our noise  can be seen as the formal space-time derivative of a centered Gaussian process whose covariance is given by
\begin{equation}  \label{cov}
\be \left[
 W(s,x)W(t,y)\right]=
\frac{1}{2}\left( |x|^{2H}+|y|^{2H}-|x-y|^{2H}\right) \, (s\wedge t),
\end{equation}
where $\frac 14<H<\frac{1}{2}$.   That is, $W$ is a standard Brownian motion in time and a
fractional Brownian motion with Hurst parameter $H$  in  the space variable. Notice that the spatial covariance of $\dot{W}$, which is formally given by $\ga(x)=H(2H-1) |x-y|^{2H-2}$, is not locally integrable when $H<\frac 12$. It is in fact a nonpositive distribution, and therefore the stochastic integration with respect to $W$ cannot be handled by classical theories (see e.g.  \cite{DPZ, Da, Wa}). However, one  has recently been able (cf. \cite{HHLNT}) to give a proper definition of equation \eqref{eq:spde-pam} and to solve it in a space of H\"older continuous processes (see also the recent work \cite{BJQ}, covering the linear case \eqref{eq:spde-pam}). We shall take those results for granted.

Let us now highlight the fact that space-time asymptotics for stochastic heat equations like \eqref{eq:spde-pam} have attracted a lot of attention in the recent past. This line of research stems from different motivations, and among them let us quote the following.  For a fixed $t>0$, the large scale behavior of the function $x\mapsto u(t,x)$ is dramatically influenced by the presence of the noise $\dot{W}$ in \eqref{eq:spde-pam} (as opposed to a deterministic equation with no noise). One way to quantify this assertion is to analyze the asymptotic behavior of $x\mapsto u(t,x)$ as $|x|\to\infty$. Results in this sense include intermittency results, upper and lower bounds for $M_{R}\equiv\sup_{|x|\le R} u(t,x)$ contained in \cite{CJK}, and culminate in the sharp results obtained in \cite{Ch3}. Roughly speaking, in case of a white noise in time like in \eqref{cov}, those articles establish that $\log(M_{R})$ behaves like $[\log(R)]^{\psi}$, for an exponent $\psi$ which depends on the spatial covariance structure of $\dot{W}$. In particular, if the spatial covariance of $\dot{W}$ is described by the Riesz kernel $|x|^{-\al}$ for $\al\in(0,1)$, one gets $\psi=\frac{2}{4-\al}$. This interpolates  between a regular situation in space ($\al=0$ and $\psi=1/2$) and the white noise setting ($\al=1$ and $\psi=2/3$). In any case those results are in sharp contrast with the deterministic case, for which $x\mapsto u(t,x)$ stays bounded.

With these preliminaries in mind, the current contribution completes the space-time asymptotic picture for the stochastic heat equation, covering very rough situations like the ones described by \eqref{cov}. Namely,  we shall get the following spatial asymptotics.

\begin{theorem}\label{thm:space-asymptotics}
Let $W$ be the Gaussian  field given by the covariance \eqref{cov}, and assume $\frac14<H<\frac12$.
Let $u$ be the unique solution to equation~\eqref{eq:spde-pam} driven by $\dot{W}$ with initial condition $u_0=1$ (see Proposition~\ref{prop:exist-unique-FK} for a precise statement), and consider $t>0$. Then
\begin{equation}\label{eq:as-limit-in space}
\lim_{R\to\infty}(\log R)^{-{\frac{1}{1+H}}}\log\lp \max_{\vert x\vert\le R}u(t,x) \rp
= c_0(H)(t \mathcal{E})  ^{H\over 1+H} \  \hskip.2in \text{a.s.},
\end{equation}
where $\mathcal E$ is the variational constant to be defined below in Proposition \ref{lem:max-E-theta} and $c_0(H)$ is a constant depending only on $H$, given by
\begin{equation}  \label{c0}
c_0(H)= (1+H) \lp \frac{c_{H}}{2} \rp^{\frac{1}{1+H}} \lp  \frac{1}{H} \rp^{\frac{H}{1+H}},
\end{equation}
with $c_H$ defined in  (\ref{eq:expr-c1H}).
\end{theorem}

Let us say a few words about our strategy in order to prove Theorem \ref{thm:space-asymptotics}. It can be roughly be divided in two main steps:

\noindent
\textit{(i) Tail estimate for $u(t,x)$.}
Let us fix $t\in\R_{+}$ and $x\in\R$. We will see (cf Corollary \ref{c.2.8}) that for large $a$, we have
\begin{equation}\label{eq:equiv-tail}
\PP\lp \log( u(t,x))\ge   a\rp
\simeq 
\exp\lp -\frac{\hat{c}_{H,t} \, a^{1+H}}{t^{H}}     \rp,
\end{equation}
where $\hat{c}_{H,t}$ is determined by a variational problem. This stems, via some large deviation arguments, from a sharp analysis of the high moments of $u(t,x)$. Namely, our main effort in order to get the tail behavior will be to prove that for large $m\in\N$, we have (see Theorem~\ref{thm:asymptotic-moments-utx})
\begin{equation}\label{eq:equiv-moments}
\EE\lc \lp u(t,x) \rp^{m}  \rc \simeq 
\exp\lp c_{H} t m^{1+\frac{1}{H}}  \rp,
\end{equation}
with a variational expression for $c_{H}$.
Towards this aim, we resort to a Feynman-Kac representation for the moments of $u(t,x)$, which involves a kind of intersection local time for an $m$-dimensional Brownian motion weighted by a singular potential. We are thus able to relate the quantity $\EE[ ( u(t,x) )^{m}  ]$ to a semi-group on $L^{2}(\R^{m})$, and this semi-group admits a generator $A_{m}$ which can be expressed as the Laplace operator on $\R^{m}$ perturbed by a singular distributional potential. Then we shall get our asymptotic result \eqref{eq:equiv-moments} thanks to a careful spectral analysis of $A_{m}$. The technicalities related to this step are detailed in Sections \ref{sec:optim-prob} and~\ref{sec:moment-asympt}.

\noindent
\textit{(ii) Spatial behavior.}
Once the tail of $\log(u(t,x))$ has been sharply estimated, we can complete the study of the asymptotic behavior in the following way: on the interval $[-M,M]$  for large $M$, we are able to produce some random variables $X_{1},\ldots,X_{\cn}$ such that:
\begin{itemize}
\item
$\cn$ is of order $2M$.
\item 
$X_{1},\ldots,X_{\cn}$ are i.i.d, and satisfy approximately \eqref{eq:equiv-tail}.
\item
$X_{1},\ldots,X_{\cn}$ are approximations of $u(t,x_{1}),\ldots u(t,x_{\cn})$ with $x_{1},\ldots,x_{\cn}\in[-M,M]$.
\item
Fluctuations of $u$ in small boxes around $x_{1},\ldots,x_{\cn}$ are small.
\end{itemize}
With this information in hand, the behavior $(\log R)^{\frac{1}{1+H}}$ in Theorem \ref{thm:space-asymptotics} can be heuristically understood as follows: for an additional parameter $\la$, we have
\begin{equation*}
\PP\lp \max_{j\le \cn} \log( |X_{j}|) \le \la [\log R]^{\frac{1}{1+H}} \rp
=
\lc 1 - \PP\lp  \log(|X_{j}|) > \la [\ln(R)]^{\frac{1}{1+H}} \rp  \rc^{\cn},
\end{equation*}
and thanks to the tail estimate \eqref{eq:equiv-tail}, we obtain
\begin{equation*}
\PP\lp \max_{j\le \cn} \log(|X_{j}|) \le \la [\log R]^{\frac{1}{1+H}} \rp
\simeq
\lc 1 - \exp\lp -\hat{c}_{H,t} \la^{1+H} \log R  \rp  \rc^{\cn}.
\end{equation*}
With some elementary calculus considerations, and playing with the extra parameter $\la$, one can now easily check that for $R$ large enough
\begin{equation*}
\PP\lp \max_{j\le \cn} \log(|X_{j}|) \le \la [\log R]^{\frac{1}{1+H}} \rp
\le \exp(-R^{\nu}),
\end{equation*}
with a positive $\nu$. Otherwise stated, we obtain an exponentially small probability of having $\log(|X_{j}|)$ of order less than $[\log R ]^{\frac{1}{1+H}}$. Using a Borel-Cantelli type argument and the fact that fluctuations of $u$ in small boxes around $x_{1},\ldots,x_{\cn}$ are small, we thus prove Theorem~\ref{thm:space-asymptotics}. 

As already mentioned, the spatial covariance $\ga$ of the noise $\dot{W}$ driving equation~\eqref{eq:spde-pam} is a nonpositive distribution. With respect to smoother cases such as the ones treated in~\cite{Ch3}, this induces some serious additional difficulties which can be summarized as follows. First, the variational asymptotic results involving the generator $A_{m}$ cannot be reduced to a one-dimensional situation due to the singularities of $\ga$. We thus have to handle a family of optimization problems in $L^{2}(\R^{m})$ for arbitrarily large $m$. Then, still in the part concerning the asymptotic behavior of $m\mapsto\EE[u(t,x)]$, the upper bound obtained  in~\cite{Ch3} relied heavily on a  compactfication by folding argument for which the positivity of $\ga$ was essential. 
This approach is no longer applicable here, and we have to replace it by a coarse graining procedure.
Finally, the localization procedure and the study of fluctuations in the spatial behavior step of our proof, though similar in spirit to the one in Conus et al. \cite{CJKS}, is more involved in its implementation. More specifically, in our case the moment estimates cannot be obtained by using sharp  Burkholder inequalities, because of the roughness of the noise. For this reason we use Wiener chaos expansions and hypercontractivity, which are more suitable methods in our context. The fluctuation estimates alluded to above are also obtained through chaos expansions.

 The paper is organized as follows. Section \ref{sec:prelim} contains some preliminaries on stochastic integration with respect to the rough noise $\dot{W}$ and the mild formulation of equation  (\ref{eq:spde-pam}). We introduce the variational quantities and their asymptotic behavior when time is large in Section \ref{sec:optim-prob}. Section \ref{sec:fk-semigroups} deals with Feynman-Kac semigroups and in Section \ref{sec:moment-asympt} we derive the precise moments asymptotics which are required to show Theorem \ref{thm:space-asymptotics}. The proof of Theorem \ref{thm:space-asymptotics} is given in Section \ref{sec:proof}.  A technical lemma is proved in the appendix.

\section{Multiplicative stochastic heat equation}\label{sec:prelim}

This section is devoted to recall the basic existence and uniqueness results for the stochastic equation with rough space-time noise. 

\subsection{Structure of the noise}\label{sec:noise-structure}

Recall that we are considering a Gaussian  field $W$ whose covariance structure is given by \eqref{cov}. As mentioned above, the stochastic integration with respect to $W$ has only been introduced recently in \cite{BJQ,HHLNT}, and we proceed now to a brief review of the results therein.

Let us start by introducing our basic notation on Fourier transforms
of functions. The space of   Schwartz functions is
denoted by $\mathcal{S}$. Its dual, the space of tempered distributions, is $\mathcal{S}'$.  The Fourier
transform of a function $g \in \mathcal{S}$ is defined by
\[ \mathcal{F}g ( \xi)  = \int_{\mathbb{R}} e^{- i
   \xi  x } g ( x) d x, \]
so that the inverse Fourier transform is given by $\mathcal{F}^{- 1} g( \xi)
= ( 2 \pi)^{- 1} \mathcal{F}g ( - \xi)$.

 Let $ \mathcal{D}((0,\infty)\times \R)$ denote the space  of real-valued infinitely differentiable functions with compact support on $(0, \infty) \times \R$.
Taking into account the spectral representation of the covariance function of the fractional Brownian motion in the case $H<\frac 12$
proved in \cite[Theorem 3.1]{PT}, we represent  our noise $W$   by a zero-mean Gaussian family $\{W(\vp), \, \vp\in
\mathcal{D}((0,\infty)\times \R)\}$ defined on a complete probability space
$(\Omega,\cf, \PP)$, whose covariance structure
is given by
\begin{equation}\label{eq:cov1}
\be\lc W(\vp) \, W(\psi) \rc
=  c_{H}\int_{\R_{+}\times\R}
\cf\varphi(s,\xi) \, \overline{\cf\psi(s,\xi)} \, \mu(d\xi) \, ds  ,
\end{equation}
where the Fourier transforms $\cf\varphi,\cf\psi$ are understood as Fourier transforms in space only and  where we have set
\begin{equation}\label{eq:expr-c1H}
c_{H}= \frac 1 {2\pi} \Gamma(2H+1)\sin(\pi H)  ,
\quad\text{and}\quad
\mu(d\xi)= |\xi|^{1-2H} d\xi \,.
\end{equation}
We denote by $\gamma$ the Fourier transform of the measure
$\mu(d\xi)$.   Formally $c_H \gamma(x) \delta_0(s)$ is the covariance function of the generalized noise $\dot{W}$.
However, notice that  $\gamma$ is a generalized function  and  the integral $\gamma(x)=\int_{\R}e^{-i\xi x}\mu(d\xi)
$ is not defined pointwise.  Rather, it is defined as a linear
functional  given by
\[
\int_{\R}\gamma(x)\varphi(x)dx
=\int_{\R}{\mathcal F}\varphi(\xi)\mu(d\xi),
\]
for any $\varphi \in {\mathcal S}(\R)$.
 As a generalized
function, $\gamma$ is non-negative definite in the sense that
$$
\int_{\R\times\R}\gamma(x-y)\varphi(x)\varphi(y)dx dy
=\int_{\R}\vert{\mathcal F}\varphi(\xi)\vert^2\mu(d\xi)
\ge 0,
\hskip.2in \varphi\in {\mathcal S}(\R).
$$
On the other hand, $\gamma (x)$ (or more precisely, its truncated form)
takes negative values somewhere. As mentioned in the introduction, this fact makes the problem of spatial 
asymptotics for equation
\eqref{eq:spde-pam} much harder than in \cite{Ch3}.

Let $\ch$ be the closure of $\mathcal{D}((0,\infty)\times \R)$ under the semi-norm induced by the right-hand side of \eqref{eq:cov1}.    The Gaussian family $W$ can be extended as an isonormal  Gaussian process $W=\{W(\varphi), \varphi\in  \ch\}$  indexed by  the Hilbert space $\ch$. The space $\ch$ can be identified with
the homogenous Sobolev space of order $\frac 12-H$ of functions with values in $L^2(\R_+)$, namely $\ch= \dot{H} ^{\frac 12- H} (L^2(\R_+))$  (see  \cite{BCD} for the definition of $\dot{H} ^{\frac 12- H}$).

Let us close this subsection by the definition of It\^o's type integral in our context, which will play a crucial role in the sequel. We will make use of the notation for any $t\ge  0$ and $\varphi\in \mathcal{S}(\R)$
\begin{equation} \label{notation}
W(t,\varphi)= W(\mathbf{1}_{[0,t]}\varphi).
\end{equation}
For each $t\ge 0$, we denote by $\mathcal{F}_t$ the $\sigma$-field generated by the random variables $\{W(s, \varphi): s\in [0,t], \varphi \in \mathcal{S}(\R)\}$. 
The  following  proposition is borrowed from \cite{HHLNT}.

\begin{proposition}\label{prop:ito-integration}
Let $L^2_a$ be  the space of predictable processes $g$ defined on $\R_{+}\times\R$ such that
almost surely $g\in \ch$ and $\be[\|g\|_{\ch}^{2}]<\infty$. Then, the stochastic integral   $\int_{\mathbb{R}_+}\int_{\mathbb{R}}g(s,x) \, W(ds,dx)$ is well defined for   $g\in L^2_a$. Furthermore, the following isometry property holds true
\begin{eqnarray}\label{int isometry}
\be\lc \lp \int_{\mathbb{R}_+}\int_{\mathbb{R}}g(s,x) \, W(ds,dx) \rp^{2} \rc
&=&
\be \left[ \|g\|_{\ch}^{2}\right]  \\
&=&
c_{H}\int_{\RR_+\times \RR} \be\lln \lc \mathcal{F}g(s,\xi)\rrn^{2} \rc |\xi|^{1-2H }d\xi ds\,. \notag
\end{eqnarray}
\end{proposition}

\subsection{Stochastic heat equation with rough multiplicative noise}

Recall that we are considering equation \eqref{eq:spde-pam} driven by the noise described in Section  \ref{sec:noise-structure}. For the sake of simplicity, we shall moreover choose $u(0,\cdot)=\1$ as the initial condition.

\begin{definition}\label{def-sol-sigma}
Let $u=\{u(t,x),  t\ge 0, x \in \mathbb{R}\}$ be a real-valued predictable stochastic process. Assume that for all $t\ge 0$ and $x\in\R$ the process $\{p_{t-s}(x-y)u(s,y) \1_{[0,t]}(s), 0 \leq s \leq t, y \in \mathbb{R}\}$ is an element of $L^2_a$, where $p_t(x)$ is the heat kernel on the real line related to $\frac{1}{2}\Delta$ and $L^2_a$ is defined in Proposition \ref{prop:ito-integration}. We say that $u$ is a mild solution of \eqref{eq:spde-pam} if for all $t \ge 0$ and $x\in \mathbb{R}$ we have
\begin{equation}\label{eq:mild-time-dep}
u(t,x)=1+ \int_0^t \int_{\mathbb{R}}p_{t-s}(x-y)u(s,y) W(ds,dy) \quad a.s.,
\end{equation}
where the stochastic integral is understood in the It\^o sense of Proposition \ref{prop:ito-integration}.
\end{definition}

Let $\{B_{j},\, j\ge 1\}$ be a collection of independent standard Brownian motions, all independent of $W$. For all $t\ge 0$ and $j<k$, we can define (see \cite{HHLNT} again) the  functional
\begin{equation} \label{y1}
V_{j,k}(t):= \int_0^t   \gamma(B_j(s)- B_k(s) )ds,
\end{equation}
which is  interpreted as  the following  limit in
  $L^2(\Omega)$
\[
V_{j,k}(t) =\lim_{\varepsilon \rightarrow 0} \int_0^t  \int_{\R}  e^{-\varepsilon |\xi|^2}   e^{ i\xi  (B_j(s)- B_k(s))}
\mu(d\xi) ds.
\]
With these notations in mind, let us quote an existence and uniqueness result for our equation of interest.

\begin{proposition}\label{prop:exist-unique-FK}
There is a unique nonnegative mild solution $u$ to  equation \eqref{eq:mild-time-dep}, understood as in Definition \ref{def-sol-sigma}. Moreover, recalling our notation \eqref{y1} above, we have:

\noindent
\emph{(i)}
The following Feynman-Kac formula for the moments of $u$ holds true for $m\ge 2$
\begin{equation}  \label{y3}
\be  \lc \lp u(t,x) \rp^{m}  \rc= \be_{x} \left[ \exp\left(  c_{H}  Q_m(t)\right)\right],
\quad\text{with}\quad
Q_m(t)=  \sum_{1\le j <k \le m} V_{j,k}(t),
\end{equation}
where    $\{B_j, j=1,\dots, m\}$ is a family of independent standard Brownian motions starting from $x\in\R$, and $\be_{x}$ denotes the expected value with respect to the Wiener measure shifted by $x$.

\noindent
\emph{(ii)} For any $m\ge 1$ and any $\al>0$ there exist some constants $c_1$ and $c_2$ such that
\begin{equation}\label{eq:ineq-moments-u}
 \be  \lc \exp\lp \al Q_{m}(t) \rp  \rc
 \le c_{3} \exp\left( c_{4} m^{1+\frac 1H} t \right).
\end{equation}
In particular, for $t\ge 0$ and $x\in\R$ we have
\begin{equation*}
 \be  \lc \lp u(t,x) \rp^{m}  \rc
 \le c_{5} \exp\left( c_{6} m^{1+\frac 1H} t \right).
\end{equation*}

\end{proposition}
One of the main steps in our estimates will be to obtain a sharp refinement of inequality~\eqref{eq:ineq-moments-u}.

\section{Preliminaries on Dirichlet forms and semigroups}\label{sec:optim-prob}

As mentioned in the Introduction, the Dirichlet form and the semigroup related to a certain operator $A_{m}$ will play a prominent role in our analysis of the spatial behavior of $u$. The current section defines an analyzes these objects.

\subsection{Variational quantities}\label{sec:var-qty}

We will see in Section \ref{sec:moment-asympt} that our sharp asymptotic estimates rely on an optimization problem for some variational quantities related to the coefficients of equation \eqref{eq:spde-pam}. We now derive some analytic properties of those quantities.

\subsubsection{A variational form on $\R$}

Let us consider the following general problem: let $\ck_{1}$ be the space defined by
\begin{equation}\label{eq:def-K1}
{\mathcal K}_{1}=\big\{g\in{L}^2(\R): \hskip.1in \|g\|_2=1\hskip.05in
\hbox{and}\hskip.05in g'\in{L}^2(\R)\big\}.
\end{equation}
Next, for a given parameter $\te>0$ and $g\in\ck_{1}$ set
\begin{equation}\label{eq:def-H-theta}
H_{\te}(g)
=
\theta \, \int_\R  \lln \cf g^{2}(\xi) \rrn^{2} \vert \xi\vert^{1-2H}d\xi
-\frac{1}{2}\int_{\R}\vert g'(x)\vert^2 dx .
\end{equation}
We are interested in optimizing this kind of variational quantity, and here is a first result in this direction.

\begin{proposition}\label{lem:max-E-theta}
For $\theta>0$  and $g\in\ck_{1}$ consider the variational quantity $H_{\te}(g)$ defined by~\eqref{eq:def-H-theta} and set
\begin{equation}\label{eq:def-e-theta}
\ce_{\theta} := \sup\lcl  H_{\te}(g) :\,  g\in {\mathcal K}_{1} \rcl.
\end{equation}
Then the following holds true:

\noindent
\emph{(i)} The quantity $\ce_{\theta}$ is finite for any $\theta>0$.

\noindent
\emph{(ii)} Setting $\ce = \ce_1$, we have
  $\mathcal{E}_\theta = \theta^{\frac 1H} \mathcal{E}$.
\end{proposition}

\begin{proof}
Let us first focus on item (i).
For any $g\in\ck_{1}$ and $\xi\in\R$ we have
$$
\lln \cf g^{2}(\xi) \rrn
=
\bigg\vert\int_\R e^{-i\xi x}g^2(x)
dx\bigg\vert
\le
\int_\R  \lln g^2(x) \rrn
dx
=
1.
$$
In addition, an elementary integration by parts argument shows that
$$
\int_\R e^{-i\xi x}g^2(x) \, dx
=
- i \int_\R\left(\frac{1}{\xi}\, \frac{dg^2}{ dx}\right)
e^{- i \xi x} \, dx.
$$
Hence, for any $\xi\in\R$ and $g\in\ck_{1}$ we get
$$
\lln \cf g^{2}(\xi) \rrn
\le
\vert\xi\vert^{-1}\int_\R\Big\vert{dg^2\over dx}(x)\Big\vert dx
=2\vert\xi\vert^{-1}\int_\R\vert g(x)\vert \vert g'(x)\vert dx
\le 2\vert\xi\vert^{-1}\|g'\|_2,
$$
where the last relation is due to Cauchy-Schwarz'  inequality plus the fact that $\|g\|_{L^{2}}=1$ for $g\in\ck_{1}$.
Consider now an additional parameter $R>0$. Gathering the two bounds we have obtained for  $\lln \cf g^{2}(\xi) \rrn$, we end up with
\begin{equation*}
\int_\R \lln \cf g^{2}(\xi) \rrn^{2} \, \vert\xi\vert^{1-2H}d\xi
\le\int_{-R}^R\vert \xi\vert^{1-2H}d\xi
+4\|g'\|_2^2\int_{\{\vert \xi\vert\ge R\}}\vert\xi\vert^{-(1+2H)}d\xi 
=\frac{R^{2-2H}}{1-H}+ \frac{4 \|g'\|_2^2}{H \, R^{2H}}.
\end{equation*}
We thus take $R=R_{\te}$ large enough, such that $\frac{4 \, \te}{H \, R^{2H}}\le\frac12$. We get
$$
\theta\, \int_\R \lln \cf g^{2}(\xi) \rrn \, \vert\xi\vert^{1-2H}d\xi
-\frac{1}{ 2}\int_{\R}\vert g'(x)\vert^2dx
\le \frac{\theta \, R^{2-2H}}{ 1-H}.
$$
Since the quantity $R_{\te}$ does not depend on $g$, the above inequality is valid for all $g\in{\mathcal K}_{1}$. The proof of item (i) is thus finished.

In order to check item (ii), consider a parameter $a>0$, and for $g\in\ck_{1}$ set $g_{a}(x)=a^{1/2}g(ax)$. It is readily checked that $g_{a}\in\ck_{1}$ whenever $g\in\ck_{1}$. In addition, we have 
\begin{equation*}
\cf g_{a}^{2}(\xi)= \cf g^{2}\!\lp \frac{\xi}{a} \rp ,
\quad\text{and}\quad
g_{a}'(x) = a^{3/2} g'(ax).
\end{equation*}
Plugging this information into the definition \eqref{eq:def-H-theta}, we obtain
\begin{equation*}
H_{\te}(g_{a})
=
\theta a^{2-2H} \, \int_\R  \lln \cf g^{2}(\xi) \rrn^{2} \vert \xi\vert^{1-2H}d\xi
-\frac{a^{2}}{2}\int_{\R}\vert g'(x)\vert^2 dx .
\end{equation*}
We now choose $a=\te^{\frac{1}{2H}}$, which yields $H_{\te}(g_{a})=\te^{\frac{1}{H}}H_{1}(g)$. Our point (ii) is now easily derived. 
\end{proof}

\begin{remark}
Notice that  Proposition  \ref{lem:max-E-theta} holds for any $H\in (0,\frac 12)$.
\end{remark}

\begin{remark}
The variational quantity $H_{\te}$ is defined in \eqref{eq:def-H-theta} appealing to Fourier coordinates. We could also have tried to introduce it in direct coordinates as
\begin{equation}\label{eq:variation-blow}
\hat{H}_{\te}(g) = 
\int_{\R\times\R}
{g^2(x)g^2(y)\over\vert x-y\vert^{2-2H}} \, dxdy
-\frac12\int_{\R}\vert g'(x)\vert^2dx.
\end{equation}
However, this quantity blows up for a broad class of functions in $\ck_{1}$. Indeed, for $g\in\ck_{1}$ such that $g\ge \al>0$ on a neighborhood of 0, we have
$$
\int_{\R\times\R}
\frac{g^2(x)g^2(y)}{\vert x-y\vert^{2-2H}} \, dxdy
\ge  \al^{4} \int_{\{\vert x\vert\le\varepsilon \}\times\{\vert y\vert\le\varepsilon \}}
\frac{1}{\vert x-y\vert^{2-2H}} \, dxdy=\infty .
$$
\end{remark}

\subsubsection{A variational form on $\R^{m}$}

Fix $m\ge 2$. Our future computations also rely on the following variational quantity on $\R^{m}$: 
\begin{equation}\label{eq:def-K-theta-m}
K_{\te,m}(g) =
\frac{\theta}{m}\sum_{1\le j<k\le m}
\int_{\R^m}\gamma(x_j-x_k)g^2(x)dx
-\frac{1}{2}\int_{\R^m}\vert \nabla g(x)\vert^2
dx.
\end{equation}
Notice that $K_{\te,m}$ can also be interpreted as a Dirichlet form related to a  Schr\"odinger type generator $A_{\te,m}$, that is,
\begin{equation}\label{eq:def-H-m}
K_{\te,m}(g) = \lla A_{\te,m} g, \, g\rra_{L^{2}(\R^{m})},
\quad\text{with}\quad
A_{\te,m}=\frac{1}{2}\Delta + \frac{\theta}{m} \sum_{1\le j<k\le m}\gamma(x_j-x_k).
\end{equation}
Observe, however, that $K_{\te,m}$ and $A_{\te,m}$ are only defined for smooth test functions, due to the fact that $\ga$ is a distribution. 

\begin{remark}
The quantity $K_{\te,m}(g)$ can also be expressed in Fourier modes. Indeed, the inverse Fourier transform of $x\in\R^{m}\mapsto \ga(x_{j}-x_{k})$ is given by
\begin{equation*}
\vp_{jk}(\xi) = |\xi_{j}|^{1/2-H} \, \delta_{0}(\xi_{j}+\xi_{k}) \, \prod_{l\neq j,k} \delta_{0}(\xi_{l}).
\end{equation*}
Hence for $j<k$ we end up with
\begin{equation*}
\int_{\R^{m}} \ga(x_{j}-x_{k}) \, g^{2}(x) \, dx
=
\int_{\R^{m}} |\la|^{1-2H} \hat{g}_{jk}(\la) \, d\la,
\end{equation*}
where we define 
\begin{equation}\label{eq:def-gjk}
\hat{g}_{jk}(\la)= \cf g^{2}(0,\ldots,0,\overbrace{\la}^{j},0,\ldots,0,\overbrace{-\la}^{k},0,\ldots,0) .
\end{equation}
Summarizing, we have obtained
\begin{equation}\label{eq:K-theta-m-fourier-mode}
 K_{\te,m}(g)
=
-\frac{1}{2(2\pi)^{m}}\int_{\R^{m}} |\xi|^{2} \, \lln \cf g(\xi)\rrn^{2} \, d\xi
+ \frac{\te}{m} \sum_{1\le j < k \le m} \int_{\R} |\la|^{1-2H} \hat{g}_{jk}(\la) \, d\la.
\end{equation}

\end{remark}

\subsection{Asymptotic results for principal eigenvalues}
With those preliminary notions in hand, we now relate the principal eigenvalue of $A_{m,\te}$ with the quantity $\ce_{\te}$ following the methodology introduced in  \cite{ChPh}.

\begin{proposition}\label{prop1}  
Consider $\theta>0$ and the quantity $K_{\te,m}(g)$ given by \eqref{eq:def-K-theta-m}. Define the set $\ck_{m}$ (which is the equivalent of $\ck_{1}$ for functions defined on $\R^{m}$) as follows
\begin{equation}\label{eq:def-Km}
{\mathcal K}_m=\big\{g\in L^2(\R^m):\hskip.1in \|g\|_2=1
\hskip.05in\hbox{and}\hskip.05in  \nabla g \in L^2(\R^m)\big\}.
\end{equation}
We define the principal eigenvalue of the operator $A_{\te,m}$by
\begin{equation}\label{eq:def-principal-eigen}
\la_{\te,m}=
\sup \lcl K_{\te,m}(g);\, g\in {\mathcal  K}_{m}\rcl.
\end{equation}
Then the following asymptotic behavior holds true
\begin{equation*}
\lim_{m\to\infty}\frac{\la_{\te,m}}{m}
=\ce_{ \te / 2}
= \lp  \frac{\te}{2}\rp^{\frac{1}{H}} \ce ,
\end{equation*}
where we recall that $\ce_{\te}$ is defined by \eqref{eq:def-e-theta}.
\end{proposition}

\begin{proof}
We mostly focus on the upper bound, the lower bound being easier to obtain. To this aim we divide the proof in several steps.

\noindent
{\it Step 1: Cutoff procedure.}
Observe that the results in \cite{ChPh}  only hold for a pointwise defined function $\gamma(x)$.
For this reason we introduce the decomposition
\begin{equation}\label{eq:def-gaM1-gaM2}
\gamma(x)= \int_{-M}^Me^{-i\xi x}\mu(d\xi)
+ \int_{\R\setminus [-M, M]}e^{-i\xi x}\mu(d\xi)
=\gamma_{M}^{1}(x)+\gamma_{M}^{2}(x),
\end{equation}
where the above identity (namely the second term $\gamma_{M}^{2}(x)$)
is understood in the distribution sense. Also notice that for $j=1,2$, the function $\gamma_{M}^{j}$ can be seen as the Fourier transform of the measure $\mu_{M}^{j}$, where $\mu_{M}^{1}$ and $\mu_{M}^{2}$ are defined as follows:
\begin{equation}\label{eq:def-mu1-mu2}
\mu_{M}^{1}(d\xi) = \vert\xi\vert^{1-2H} \1_{[-M,M]}(\xi) \, d\xi,
\quad\text{and}\quad
\mu_{M}^{2}(d\xi) = \vert\xi\vert^{1-2H} \1_{[-M,M]^{c}}(\xi) \, d\xi.
\end{equation}
Then, for any $\delta \in (0,1)$, we can write
\begin{equation}\label{eq:dcp-K-theta-m}
\sup_{g\in { \mathcal  K}_{m}} K_{\te,m}(g) \le  B_{m,M}^{1} +B_{m,M}^{2},
\end{equation}
where
 \[
 B_{m,M}^{1}= \sup_{g\in { \mathcal  K}_{m}}\bigg\{\frac{\theta}{m}\sum_{1\le j<k\le m}
\int_{\R^m}\gamma_{M}^{1}(x_j-x_k)g^2(x)dx
-{1-\delta\over 2}\int_{\R^m}\vert \nabla g(x)\vert^2
dx\bigg\}
\]
and
\[
B_{m,M}^{2}=
 \sup_{g\in { \mathcal  K}_{m}}\left\{\frac{\theta}{m}\sum_{1\le j<k\le m}
\int_{\R^m}\gamma_{M}^{2}(x_j-x_k)g^2(x)dx
-{\delta\over 2}\int_{\R^m}\vert \nabla g(x)\vert^2
dx\right\}.
\]
The   term $ B_{m,M}^{1}$ is handled by  \cite{ChPh} and we obtain
\begin{equation}\label{eq:limsup-BmM-1}
  \lim_{m\to\infty}\frac{1}{m} B_{m,M}^{1}
  = \mathfrak{E}_{M,\delta,\te},
  \quad\text{with}\quad
\lim_{M\to\infty} \mathfrak{E}_{M,\delta,\te}
= 
\lp\frac{\te}{2}\rp^{\frac{1}{H}}  (1-\delta)^{1-\frac 1H}   \mathcal{E},
 \end{equation}
where the limiting behavior for $\mathfrak{E}_{M,\delta,\te}$ is a direct consequence of Proposition \ref{lem:max-E-theta}.

  \noindent
{\it Step 2: Upper bound for $B_{m,M}^{2}$.}
 We claim that the following inequality holds
  \begin{equation}  \label{eq2-1}
  B_{m,M}^{2} \le \frac {m-1}2 \sup_{g\in { \mathcal  K}_{1}}\left\{ \theta  \int_{\R }\gamma_{M}^{2}(x)g^2(x)dx
-{\delta\over 2}\int_{\R }\vert  g'(x)\vert^2
dx\right\}.
  \end{equation}
  To show this inequality, we write
\begin{eqnarray*}
&& \sum_{1\le j<k\le m}
\int_{\R^m}\gamma_{M}^{2}(x_j-x_k)g^2(x_1,\ldots, x_m)dx_1\cdots dx_m\\
&=& \frac{1}{2}\sum_{j=1}^m \sum_{k:\hskip.03in k\not=j}
\int_{\R^m}\gamma_{M}^{2}(x_j-x_k)g^2(x_1,\ldots, x_m)dx_1\cdots dx_m,
\end{eqnarray*}
and
$$
\int_{\R^m}\vert\nabla g(x)\vert^2 d x=\sum_{j=1}^m\|\nabla_jg\|_2^2
={1\over m-1}\sum_{j=1}^m\sum_{k:\hskip.03in k\not=j}\|\nabla_kg\|_2^2.
$$
Replacing  the supremum of the sum by a sum of supremums in the expression defining $B_{m,M}^{2}$,  we obtain
\begin{eqnarray}\label{eq:bnd-BmM-DmM}
B_{m,M}^{2}&\le  &{1\over 2m}\sum_{j=1}^m
\sup_{g\in { \mathcal K}_{m}}\bigg\{ \theta\sum_{k:\hskip.03in k \not= j}
\int_{\R^m}\gamma_{M}^{2} (x_j-x_k)g^2(x)dx
-{m \delta\over m-1}\sum_{k:\hskip.03in k \not= j}\|\nabla_kg\|_2^2\bigg\}  \notag
\\
&=&\frac{1}{2}\sup_{g\in { \mathcal K}_{m}} D_{m,M}(g),
\end{eqnarray}
where we have set
\begin{equation}\label{eq:def-DmM}
D_{m,M}(g)= \theta\sum_{k=2}^m
\int_{\R^m}\gamma_{M}^{2}(x_k-x_1)g^2(x)dx
-{m\delta\over m-1}\sum_{k=2}^m\|\nabla_kg\|_2^2.
\end{equation}
We now proceed to estimate the terms $\int_{\R^m}\gamma_{M}^{2}(x_k-x_1)g^2(x)dx$ and $\sum_{k=2}^m\|\nabla_kg\|_2^2$ above.

In order to estimate the terms $\int_{\R^m}\gamma_{M}^{2}(x_k-x_1)g^2(x)dx$, notice that
\begin{eqnarray*}
&& \int_{\R^{m-1}}\gamma_{M}^{2}(x_k-x_1)g^2(x_1, \dots, x_m )dx_2 \cdots dx_m \\
&=&\int_{\R^{m-1}}\gamma_{M}^{2}(x_k)g^2(x_1, x_2+x_1\dots, x_m+x_1)dx_2\cdots dx_m\\
&=&\int_{\R^{m-1}}\gamma_{M}^{2}(x_k)\tilde{g}^2(x_1, x_2\dots, x_m)dx_2\cdots dx_m,
\end{eqnarray*}
where the function $\tilde{g}:\R^{m}\to\R$ is defined by
$$
\tilde{g}(x_1, x_2\dots, x_m)=g(x_1, x_2+x_1\dots, x_m+x_1).
$$
Now observe that $\tilde{g}$ belongs to  the space ${ \mathcal K}_{m}$ defined by \eqref{eq:def-K-theta-m}. We thus  obtain
\begin{equation*}
\sum_{k=2}^m
\int_{\R^m}\gamma_{M}^{2}(x_k-x_1)g^2(x_1,\dots, x_m) \, dx
=\sum_{k=2}^m
\int_{\R^m}\gamma_{M}^{2}(x_k)\tilde{g}^2(x_1,\dots, x_m) \, dx.
\end{equation*}
As far as the terms $\|\nabla_kg\|_2^2$ in the definition of  $D_{m,M}(g)$ are concerned, we just notice that
$\|\nabla_k\tilde{g}\|_2^2=\|\nabla_k g\|_2^2$ for every $2\le k\le m$.
Hence, recalling the definition \eqref{eq:def-DmM} of $D_{m,M}(g)$, for any $g\in\ck_{m}$ we end up with
\begin{eqnarray*}
D_{m,M}(g)
&=&\theta \sum_{k=2}^m
 \int_{\R^m}\gamma_{M}^{2}(x_k)\tilde{g}^2(x)dx
-{m\delta\over m-1}\sum_{k=2}^m\|\nabla_k\tilde{g}\|_2^2\\
&\le & \sup_{g\in{ \mathcal K}_{m}}\bigg\{\theta\sum_{k=2}^m
\int_{\R^m}\gamma_{M}^{2}(x_k)g^2(x)dx
-{m\delta\over m-1}\sum_{k=2}^m\|\nabla_kg\|_2^2\bigg\} \, ,
\end{eqnarray*}
and thus
\begin{equation}\label{eq:bnd-DmM-g}
D_{m,M}(g) \le 
\sup_{g\in{ \mathcal K}_{m}}\bigg\{\theta\sum_{k=2}^m
\int_{\R }\gamma_{M}^{2}(x_k)g_k^2(x_k)dx_k
-{m\delta\over m-1}\sum_{k=2}^m\|\nabla_kg\|_2^2\bigg\}\,,
\end{equation}
where for $k=2,\dots, m$ we have set
$$
g_k(x_k )=\bigg(\int_{\R^{m-1}}g^2(x_1,\dots, x_m)
\prod_{1\le j \le m, \, j \not = k} dx_j\bigg)^{1/2}.
$$

Let us further analyze the term $\|\nabla_kg\|_2^2=\|\nabla_kg\|_{L^{2}(\R^{m})}^2$ above, and relate it with $\|g_{k}^{\prime}\|_2^2=\|g_{k}^{\prime}\|_{L^{2}(\R)}^2$.
To this aim, it is readily checked that $g_k\in { \mathcal K}_1$ whenever $g\in\ck_{m}$. Furthermore, we have
\[
 g'_k(x_k)= \bigg(\int_{\R^{m-1}}g^2(x )
\prod_{j \not= k}^mdx_j\bigg)^{-\frac 12}
 \int_{\R^{m-1}}g(x )\nabla_k
g(x )\prod_{j \not =k }^mdx_j.
\]
By Cauchy-Schwarz'  inequality, we thus get
\begin{equation}\label{eq:cauchy-schwarz-1}
\vert g'_k(x_k)\vert^2\le\int_{\R^{m-1}}\vert\nabla_k
g(x )\vert^2\prod_{j \not= k}^mdx_j ,
\end{equation}
which yields
$$
\|g_{k}^{\prime}\|_2^2 
=
\int_{\R }\vert   g'_k(x_k)\vert^2 dx_k
\le\|\nabla_kg\|_2^2,  \hskip.2in k=2,\dots, m.
$$
Plugging this inequality into \eqref{eq:bnd-DmM-g}, we have obtained
\begin{equation*}
D_{m,M}(g) \le 
\sup_{g\in{ \mathcal K}_{m}}\bigg\{\theta\sum_{k=2}^m
\int_{\R }\gamma_{M}^{2}(x_k)g_k^2(x_k)dx_k
-{m\delta\over m-1}\sum_{k=2}^m\|g_{k}^{\prime}\|_2^2\bigg\}\,,
\end{equation*}
and recalling relation \eqref{eq:bnd-BmM-DmM}, this yields
\begin{eqnarray*}
B_{m,M}^{2}&\le &
\frac{1}{2}\sup_{g\in{ \mathcal K}_{m}}\bigg\{ \theta\sum_{k=2}^m
\int_{\R }\gamma_{M}^{2}(x_k)g_k^2(x_k)dx_k
-{m\delta \over m-1}\sum_{k=2}^m\|g_{k}^{\prime}\|_2^2
\bigg\}\\
&\le& {m-1\over 2}\sup_{g\in{ \mathcal K}_1}\bigg\{ \theta \int_{\R }\gamma_{M}^{2}(x)g^2(x)dx
- \delta \int_{\R }\vert  g'(x)
\vert^2 dx \bigg\},
\end{eqnarray*}
which  shows (\ref{eq2-1}).

\noindent
{\it Step 3: Limiting behavior for $B_{m,M}^{2}$.}
The estimate (\ref{eq2-1}) implies
\begin{equation}\label{eq:corol-eq2-1}
\limsup_{m\rightarrow \infty} \frac{B_{m,M}^{2}}{m} \le \frac 12\sup_{g\in{ \mathcal K}_1}\bigg\{ \theta \int_{\R }\gamma_{M}^{2}(x)g^2(x)dx
- \delta \int_{\R }\vert  g'(x)
\vert^2 dx
\bigg\}.
\end{equation}
 We now show that this variation goes to zero as $M$ tends to $\infty$. This can be seen by a simple integration by parts argument. For any $g\in\ck_{1}$ we have
$$
\cf g^{2}(\xi)
=
\int_{\R}e^{-i\xi  x}g^2(x)dx
=
-\frac{i}{\xi} \int_{\R}e^{-i\xi  x}  ( g^2)'(x)dx
=
-\frac{i}{\xi} \, [\cf  (g^2)'](\xi).
$$
Hence Plancherel's  identity yields
\begin{align}\label{eq:plancherel-gammaM-g2}
&\int_{\R }\gamma_{M}^{2}(x)g^2(x)dx 
=
\int_{\{\vert\xi\vert\ge M\}}
\lln \cf g^{2}(\xi) \rrn \mu(d\xi  )
=
\int_{\{\vert\xi\vert\ge M\}} \lln [\cf  ( g^2)'](\xi) \rrn \vert\xi\vert^{-2H}d\xi \notag \\
&\le
\bigg(\int_{\{\vert\xi\vert\ge M\}}\vert\xi\vert^{-4H}d\xi\bigg)^{1/2}
\bigg(\int_{\R} \lln [\cf ( g^2)'](\xi) \rrn^{2} d\xi\bigg)^{1/2}
=
\frac{\lp \int_{\R} \lln [\cf  ( g^2)'](\xi) \rrn^{2} d\xi \rp^{1/2}}{(4H-1)^{1/2} M^{2H-1/2}} ,
\end{align}
and it should be stressed that we are using our assumption $H>1/4$ in order to get nondivergent integrals above.
In addition, another use of Plancherel's  identity enables  us to write
\begin{equation}\label{eq:plancherel-nabla-g2}
\int_{\R} \lln [\cf  ( g^2)'](\xi) \rrn^{2} d\xi
=
2\pi \int_{\R}\vert  ( g^{2})'(x)\vert^2 dx
= 4 \pi \int_{\R} g^2(x)\vert g'(x)\vert^2 dx.
\end{equation}
We now get a uniform bound on $g^{2}$. Since $g\in\ck_{1}$ we have $\|g\|_{2}=1$ and thus
$$
g^2(x)=2\int_{-\infty}^x g(y)g'(y)dy\le 2\|g\|_2 \|g'\|_2
=2\|g'\|_2.
$$
Plugging this information into \eqref{eq:plancherel-nabla-g2} we obtain
$$
\int_{\R} \lln [\cf  ( g^2)'](\xi) \rrn^{2} d\xi
\le 4 \pi \|g'\|_2^3.
$$
Finally, recalling \eqref{eq:plancherel-gammaM-g2} we can write
\begin{equation*}
\theta\int_{\R }\gamma_{M}^{2}(x)g^2(x)dx - \delta \int_{\R} \vert g'(x)\vert^2 dx
\le
\vp(\|g'\|_{2}^{1/2}),
\end{equation*}
where $\vp(x)= a x^{3} - \delta x^{4}$ and $a=\frac{2\pi^{1/2}\te}{(4H-1)^{1/2}M^{2H-1/2}}$. Optimizing over the values of $x\ge 0$ we get a bound which is uniform over $g\in\ck_{1}$
\begin{equation*}
\theta\int_{\R }\gamma_{M}^{2}(x)g^2(x)dx - \delta \int_{\R} \vert g'(x)\vert^2 dx
\le
\frac{c\, \te^{4}}{\delta^{3} M^{2(4H-1)}},
\end{equation*}
with a (nonrelevant) constant $c=\frac{27\pi^{2}}{16}$. Plugging this uniform bound into \eqref{eq:corol-eq2-1}, we thus easily end up with
\begin{equation}\label{eq:limsup-BmM-2}
\lim_{M\to\infty}\limsup_{m\rightarrow \infty} \frac{B_{m,M}^{2}}{m} = 0.
\end{equation}

\noindent
{\it Step 4: Conclusion for the upper bound.} Let us report our estimates \eqref{eq:limsup-BmM-1} and \eqref{eq:limsup-BmM-2} into the upper bound \eqref{eq:dcp-K-theta-m}. This trivially yields
\begin{equation*}
\limsup_{m\to\infty} \frac{1}{m}\sup_{g\in { \mathcal  K}_{m}} K_{\te,m}(g)
\le
\lp\frac{\te}{2}\rp^{\frac{1}{H}} (1-\delta)^{1-\frac 1H}  \mathcal{E}.
\end{equation*}

\noindent
{\it Step 5: Lower bound.}
In order to get the lower bound, we proceed to a  direct verification,
replacing the class ${\mathcal  K}_{m}$ by the smaller class
$$
\ck_{0} \equiv \left\{  g(x)= \prod_{j=1}^mg_0(x_j), \quad  g_0\in {\mathcal K}_1\right\}.
$$
Towards this aim, we resort to the expression \eqref{eq:K-theta-m-fourier-mode} of $K_{\te,m}(g)$ in Fourier modes. Then, for $g\in\ck_{0}$, it is readily checked that $\cf g(\xi)=\prod_{k=1}^{m}\cf g_{0}(\xi_{k})$ for all $\xi=(\xi_{1},\ldots,\xi_{k})$. Invoking this relation, plus the fact that $g_{0}\in\ck_{1}$, we get
\begin{equation*}
\int_{\R^{m}} |\xi|^{2} \, \lln \cf g(\xi)\rrn^{2} \, d\xi
=
(2\pi)^{m-1} \, m \int_{\R} |\la|^{2} \lln \cf g_{0}(\la)\rrn^{2} \, d\la= (2\pi)^m m \int_{\R} | g'(x)|^2 dx, 
\end{equation*}
and recalling that the functions $\hat{g}_{jk}$ are introduced in \eqref{eq:def-gjk},
\begin{equation*}
\int_{\R} |\la|^{1-2H} \hat{g}_{jk}(\la) \, d\la
=
\int_{\R} |\la|^{1-2H} \lln \cf g_{0}^{2}(\la) \rrn^{2} d\la.
\end{equation*}
Plugging those identities into the expression \eqref{eq:K-theta-m-fourier-mode}, we get
\begin{equation*}
K_{\te,m}(g)
=
- \frac{m}{2} \int_{\R} \lp g'_0(x)  \rp^{2} \, dx
+ \frac{\te (m-1)}{2 } \int_{\R} |\la|^{1-2H} \lln \cf g_{0}^{2}(\la) \rrn^{2} d\la.
\end{equation*}
This quantity is easily related to the variational expression $H_{\te}(g_{0})$ defined by \eqref{eq:def-H-theta},
from which the identity
\begin{equation*}
\lim_{m\to\infty} \sup_{g\in\ck_{0}} \frac{K_{\te,m}(g)}{m} = \lp \frac{\te}{2} \rp^{\frac{1}{H}} \ce
\end{equation*}
is readily checked. This shows our lower bound and finishes the proof.
 \end{proof}
 
\subsection{Feynman-Kac semi-groups}\label{sec:fk-semigroups}

For $m\ge 1$  and $t\ge 0$, consider the quantity $Q_{m}(t)$ defined by~\eqref{y3}. Our moment study for the solution to \eqref{eq:spde-pam} will rely on the spectral behavior of the following semi-group acting on $L^{2}(\R^{m})$
\begin{equation}\label{eq:def-semi-group}
T_{m,t}g(x)=\be_{x}\left[g(B(t)) \exp \left(  \frac{c_{H} \, Q_{m}(t)}{m}
 \right)  \right],
\end{equation}
where $c_{H}$ is the constant  defined in (\ref{eq:expr-c1H})  and $x\in\R^{m}$ represents the initial condition for the $m$-dimensional  Brownian motion $B$ appearing in~\eqref{y3}. We now establish some basic properties of those operators.

\begin{proposition}\label{prop:fk-self-adjoint}
Let $\{T_{m,t}, \, t\ge 0\}$ be the family of operators introduced in \eqref{eq:def-semi-group}. Then the following properties hold true:

\noindent
\emph{(i)}
The family $\{T_{m,t}, \, t\ge 0\}$ defines a semi-group of bounded self-adjoint operators on the space $L^{2}(\R^{m})$.

\noindent
\emph{(ii)}
The generator of $T_{m,t}$ is given on test functions by $A_{m}=A_{c_{H},m}$ defined by \eqref{eq:def-H-m}, that is
\begin{equation*}
A_{m}g = \frac12 \Delta g +  \frac{c_{H}}{m} \sum_{1\le j<k\le m}\gamma(x_j-x_k).
\end{equation*}
There exists an integer $m_{0}\ge 1$ such that for $m\ge m_{0}$, $A_{m}$ admits a self-adjoint extension (still denoted by $A_{m}$), defined on a domain ${\rm Dom}(A_{m})\subset L^{2}(\R^{m})$.
\end{proposition}

\begin{proof}
Let us first prove the boundedness of $T_{m,t}$. To this aim, notice that for $g\in L^{2}(\R^{m})$ we have
\begin{equation*}
|T_{m,t}g(x)|^{2} = \be_{x}^{2}\left[g(B(t)) \exp \left(  \frac{c_{H} \, Q_{m}(t)}{m}  \right)  \right] .
\end{equation*}
We now apply Cauchy-Schwarz' inequality, together with relation \eqref{eq:ineq-moments-u}, in order to get
\begin{equation*}
|T_{m,t}g(x)|^{2} 
\le
\be_{x}\left[g^{2}(B(t))\right] \,
\be_{x}\left[\exp \left(  \frac{2\, c_{H} \, Q_{m}(t)}{m}  \right)  \right]
\le
c_{m,H} \, \be_{x}\left[g^{2}(B(t))\right].
\end{equation*}
Taking into account the fact that $p_{t}(x-\cdot)$ is the density of $B_{t}$, we obtain
\begin{eqnarray*}
\|T_{m,t}g\|_{L^{2}(\R^{m})}^{2}
&\le&
c_{m,H} \int_{\R^{m}} \be_{x}\lc g^{2}(B_{t}) \rc \, dx
=
c_{m,H} \int_{\R^{m}} \lp \int_{\R^{m}} p_{t}(x-y) g^{2}(y) \, dy \rp dx  \\
&=&
c_{m,H} \int_{\R^{m}} \lp \int_{\R^{m}} p_{t}(x-y)  \, dx \rp g^{2}(y) \, dy
=
c_{m,H} \|g\|_{L^{2}(\R^{m})}^{2},
\end{eqnarray*}
which proves boundedness in $L^{2}(\R^{m})$. The self-adjointness of $T_{m,t}$ is then easily derived.

It is also readily checked that the infinitesimal generator of $T_{m,t}$, acting on test functions in $\cs(\R^{m})$, is given by $A_{m}$. In addition, this operator is obviously symmetric. In order to show that it admits a self-adjoint extension, it is sufficient (thanks to the classical Freidrichs extension theorem) to show that 
\begin{equation}\label{eq:upper-bnd-Am}
\langle A_{m} g, \, g \rangle \le c \, \|g\|_{L^{2}(\R^{m})}^{2},
\end{equation}
for all $g\in\dom(A_{m})$ and for a constant $c>0$. We now prove this inequality for $m$ large enough: indeed, for all test functions we have $\langle A_{m} g, \, g \rangle= K_{m}(g)$, where we have set $K_{m}(g)=K_{c_{H},m}(g)$ and $K_{\te,m}(g)$ is defined by \eqref{eq:def-K-theta-m}. The fact that there exists an  $m_{0}\ge 1$ such that relation \eqref{eq:upper-bnd-Am} holds for $m\ge m_{0}$ is then an immediate consequence of Proposition \ref{prop1}. This finishes our proof.
\end{proof}

As in the proof of Proposition \ref{prop1}, our future considerations will also rely on a truncated version of the operators $T_{m,t}$. Let us label their definition for further use.

\begin{definition}
Consider a parameter $M>0$, and the function $\ga_{M}^{1}$ defined by \eqref{eq:def-gaM1-gaM2}. For $m\ge 1$ we introduce the following quantities, defined similarly to~\eqref{y3} but replacing $\ga$ by the smoothed function $\ga_{M}^{1}$ given by \eqref{eq:def-gaM1-gaM2}:
\begin{eqnarray}
Q_{m,M}^{1} 
&=&
\sum_{1\le j <k \le m} \int_0^t   \gamma_{M}^{1}(B_j(s)- B_k(s) )ds, \label{eq:def-QmM1} \\
\hq_{m,M}^{1} 
&=&
\sum_{1\le j ,k \le m} \int_0^t   \gamma_{M}^{1}(B_j(s)- B_k(s) )ds.  \label{eq:def-hat-QmM1}
\end{eqnarray}
Related to these Feynman-Kac functionals, we consider a family of operators $\{\htt_{m,M,t},\, t\ge 0\}$ acting on $L^{2}(\R^{m})$, indexed by a parameter $\te>0$
\begin{equation}\label{eq:def-semi-group-trunc}
\htt_{m,M,t}g(x)
=
\be_{x}\left[g(B(t)) \exp \left(  \frac{\te \, \hq_{m,M}^{1}(t)}{m}
 \right)  \right].
\end{equation}
\end{definition}

The family of operators we have just introduced enjoys the following property.

\begin{proposition}\label{prop:fk-self-adjoint-truncated}
Let $\{\htt_{m,M,t},  \, t\ge 0\}$ be the families of operators introduced in~\eqref{eq:def-semi-group-trunc}. The  conclusions of Proposition \ref{prop:fk-self-adjoint} remain true for this semi-group, with a generator $\hat{A}_{m,M}$ given  by
\begin{equation*}
\hat{A}_{m,M}g = \frac12 \Delta g +  \frac{\te}{m} \sum_{1\le j,k\le m}\gamma_{M}^{1}(x_j-x_k).
\end{equation*}
Furthermore, the following limiting behavior holds true for all $x\in\R^{m}$
\begin{equation}\label{eq:lim-FK-bounded}
\lim_{t\to\infty} \frac{\log\lp\htt_{m,M,t}\1(x)\rp}{t}
=\la_{m,M},
\end{equation}
where $\la_{m,M}$ is the principal eigenvalue of $\hat{A}_{m,M}$, defined similarly to \eqref{eq:def-principal-eigen} by
\begin{equation}\label{eq:def-la-mM}
\la_{m,M}
=
\sup\lcl
\te \sum_{j,k=1}^{m}
\int_{\R^{m}}\gamma_{M}^{1}(x_j-x_k)g^2(x)dx
-\frac{1}{2}\int_{\R^{m}}\vert \nabla g(x)\vert^{2}
dx ; \, g\in\ck_{m} \rcl,
\end{equation}
and where we recall  that $\ck_{m}$ is introduced in \eqref{eq:def-Km}.
\end{proposition}

\begin{proof}
The self-adjointness of $T_{m,M,t}$  is completely classical and left to the reader. In order to get the self-adjointness of $A_{m,M}$, we just realize that $\gamma_{M}^{1}:\R\to\R$ is the inverse Fourier transform of a function which is in $L^{2}(\R)$ with compact support. Therefore $\gamma_{M}^{1}$ admits bounded derivatives of all order. The desired self-adjointness property is thus a consequence of classical results, which are summarized e.g in \cite{Ch1}. Finally relation \eqref{eq:lim-FK-bounded} is a classical Feynman-Kac limit, for which we refer to  \cite[Theorem 4.1.6]{Ch1}.
\end{proof}

\section{Asymptotic properties of the moments}\label{sec:moment-asympt}

As in \cite{Ch4,Ch3,CJK}, the spatial asymptotics for equation \eqref{eq:mild-time-dep}  will be established thanks to a sharp estimate of  the tail of  $u(t,0)$ for $t\ge 0$ fixed. As we will see later, this can  be related to some estimates on the moments $u(t,0)$, and our next step will be to obtain the exact asymptotic behavior (as $m\to\infty$) of these moments. Before we begin with this task, we will reduce our problem thanks to a series of lemmas.

First let us observe that the generator we have considered for a proper normalization procedure in Section \ref{sec:var-qty} involves a sum of the form $\frac{1}{m}\sum_{1\le j<k\le m}\gamma(x_j-x_k)$, where we emphasize the normalization by $m$. However, the quantity we manipulate in our Feynman-Kac representation \eqref{y3} is $Q_{m}(t)$, which does not exhibit this normalizing term. We will introduce the missing normalization by a simple scaling argument.

\begin{lemma}
Let $u$ be the solution to \eqref{eq:spde-pam}. Then the moments of $u$ admit the following representation
\begin{equation}\label{eq:fk-with-scaling}
\be  \lc \lp u(t,0)\rp^{m}  \rc= \be_{0} \left[ \exp\left(  \frac{c_{H} Q_{m}(t_{m})}{m}  \right)\right],
\end{equation}
where $t_{m}=m^{1/H}t$.
\end{lemma}

\begin{proof}
Start from expression \eqref{y3} for $\be [(u(t,0))^m]$, and recall that
\begin{equation*}
Q_m(t)=  \sum_{1\le j <k \le m} \int_0^t  \int_{\R}  
e^{ i\xi  (B_j(s)- B_k(s))} \vert\xi\vert^{1-2H}d\xi ds.
\end{equation*}
Setting $v=m^{1/H}s$ in the time integral, and invoking the fact that $B_{j}(m^{-1/H}\cdot)$ is equal in law to $m^{-1/2H}B_{j}(\cdot)$, we get
\begin{equation*}
Q_m(t)\stackrel{(d)}{=} \frac{1}{m^{1/H}}
\sum_{1\le j <k \le m} \int_{0}^{t_{m}}  \int_{\R} 
e^{ i\frac{\xi  (B_j(v)- B_k(v))}{m^{1/2H}}} \vert\xi\vert^{1-2H}d\xi dv.
\end{equation*}
We now set $m^{-1/2H}\xi=\la$ in the space integral above. This easily yields an equality in law between $Q_m(t)$ and $m^{-1} Q_m(t_m)$, and thus
\begin{equation}\label{eq:FK-change-variables}
\be \lc \lp u(t,0)\rp^{m}\rc=
 \be_0 \lc\exp\left(\frac{c_H \, Q_m(t_m)}{ m}
  \right) \rc,
\end{equation}
which corresponds to our claim.
\end{proof}

We now establish a couple of simple monotonicity properties for the quantity $Q_{m}$ which will feature in the sequel. The first one is related to our regularization procedure.

\begin{lemma}\label{lem:QmM-increase-M}
Let $Q_{m,M}^{1}$ be the quantity defined by \eqref{eq:def-QmM1} and let $\al$ be a positive constant. Then, the map $M\mapsto \be_{0}[\exp(\al Q_{m,M}^{1}(t))]$ is increasing.
\end{lemma}

\begin{proof}
Consider $M>0$. Then we have
\begin{equation}\label{eq:taylor-exp-QmM}
\be_{0}\lc \exp(\al Q_{m,M}^{1}(t)) \rc
=
 \sum_{n=0}^{\infty} \frac{\al^{n}\,\be_{0}\lc\lp Q_{m,M}^{1}(t))\rp^{n}\rc}{n!}.
\end{equation}
We now show that, for any fixed $n\ge 1$, the map $M\mapsto \be_{0}[( Q_{m,M}^{1}(t)))^{n}]$ is increasing. Indeed, recalling our notation \eqref{eq:def-mu1-mu2} for $\mu_{M}^{1}$, it is readily checked that
\begin{equation}\label{eq:moments-QmM}
\be_{0}\lc\lp Q_{m,M}^{1}(t))\rp^{n}\rc
= \sum_{l=1}^{n} \sum_{1\le j_{l}<k_{l}\le m}
\int_{[0, t]^{n}}\!\!\int_{\R^{n}}
\be_{0}
\lc e^{i \sum_{l=1}^{n} \xi_l  (B_{j_{l}}(s_l)-B_{k_{l}}(s_l))  }\rc  
\, [\mu_{M}^{1}]^{\otimes n}\!(d\xi) \, ds,
\end{equation}
where we use the simple convention $d\xi=d\xi_{1}\cdots d\xi_{n}$ and $ds=ds_{1}\cdots ds_{n}$. Now notice that for all $j,k,l$ we have 
\begin{equation*}
\be_{0}  \lc e^{ i\sum_{l=1}^{m}\xi_l  (B_{j_l}(s_l)- B_{k_l}(s_l))} \rc >  0.
\end{equation*}
This easily yields the fact that $M\mapsto \be_{0}[( \al Q_{m,M}^{1}(t)))^{n}]$ is increasing. Going back to our decomposition \eqref{eq:taylor-exp-QmM}, we have thus obtained that $M\mapsto \be_{0}[\exp(\al Q_{m,M}^{1}(t))]$ is increasing. 
\end{proof}

The second monotonicity property we need concerns the dependence with respect to the initial condition for our underlying Brownian motion. 
\begin{lemma}\label{lem:monotone-init-cdt}
Recall that  $ B= \{ B(s)= (B_1(s), \dots, B_m(s)), s\ge 0 \} $ is an $m$-dimensional Brownian motion. For any $x\in \mathbb{R}^m$, recall that we use  $\be_x$  to denote the mathematical expectation with respect to $B$  with $B(0)=x$. Then the following relation is verified for all $m\ge 1$, $t\ge 0$, $\te\ge 0$ and $x\in\R^{m}$
\begin{equation*}
\be_{0}\lc \exp\left( \frac{\theta  Q_m(t)}{m}  \right) \rc
\ge  \be_{x}\lc
\exp\left( \frac{\theta  Q_m(t)}{m}  \right)\rc.
\end{equation*}
For any fixed $M>0$, the same property holds true for $Q_{m,M}^{1}(t)$ and $Q_{m,M}^{2}(t)$.
\end{lemma}

\begin{proof}
We focus on the property for $Q_{m}$, the equivalent for $Q_{m,M}^{i}$ ($i=1,2$) being shown exactly in the same way. Furthermore, resorting to   expansion \eqref{eq:taylor-exp-QmM} as in the previous proof, our claim can be reduced to show that
\begin{equation}\label{eq:monotone-moments}
\be_{0}\lc \lp Q_m(t) \rp^{n} \rc
\ge  \be_{x}\lc \lp Q_m(t) \rp^{n}\rc,
\qquad\text{for all } n\ge 1,
\end{equation}
We now focus on relation \eqref{eq:monotone-moments}.

In order to show \eqref{eq:monotone-moments}, we first decompose the moments of $Q_{m}(t)$ similarly to \eqref{eq:moments-QmM}
\begin{equation}\label{eq:dcp-moments-Qm}
\be_{x}\lc\lp  Q_m(t) \rp^{n} \rc
=
\int_{[0, t]^n}\!\!\int_{(\R^m)^n}
\be_{0}\lc D_{m,n}\rc \, [\mu_{M}^{1}]^{\otimes n}\!(d\xi) \, ds,
\end{equation}
where we have set 
$D_{m,n}=\prod_{l=1}^n \sum_{1\le j_{l}<k_{l}\le m}\ e^{ i\xi_l(x_j-x_k)} e^{i \xi_l  (B_j(s_l)-B_k(s_l))  }$.
Furthermore, one can reorder terms in the quantity $D_{m,n}$, and obtain
\begin{equation*}
D_{m,n}=
\sum_{j_1,\dots, j_n, k_1, \dots, k_n=1 \atop j_l < k_l, \forall l=1,\dots, n}^mC(j_1,\dots, j_n, k_1, \dots, k_n)
e^{  i\sum_{l=1}^m\xi_l  (B_{j_l}(s_l)- B_{k_l}(s_l))},
\end{equation*}
where the constants $C(j_1,\dots, j_n, k_1, \dots, k_n)$ satisfy $|C(j_1,\dots, j_n, k_1, \dots, k_n)|=1$. It should also be noticed that
\begin{equation}\label{eq:ident-caract-fct}
\sum_{j_1,\dots, j_n, k_1, \dots, k_n=1 \atop j_l <k_l, \forall l=1,\dots, n}^m
e^{ i\sum_{l=1}^m\xi_l  (B_{j_l}(s_l)- B_{k_l}(s_l))}
=\prod_{l=1}^n
\sum_{1\le j<k\le m}
e^{i\xi_l  (B_j(s_l)-B_k(s_l))}.
\end{equation}
Hence, taking the mathematical expectation yields
\begin{eqnarray*}
\be_{0}\lc D_{m,n}\rc&=&
\Bigg|\sum_{j_1,\dots, j_n, k_1, \dots, k_n=1 \atop j_l <k_l, \forall l=1,\dots, n}^m
C(j_1,\dots, j_n, k_1, \dots, k_n)
\be_{0}\lc  e^{  i\sum_{l=1}^m\xi_l  (B_{j_l}(s_l)- B_{k_l}(s_l))} \rc \Bigg| \\
&\le&
\sum_{j_1,\dots, j_n, k_1, \dots, k_n=1 \atop j_l <k_l, \forall l=1,\dots, n}^m
\be_{0}  \lc e^{ i\sum_{l=1}^m\xi_l  (B_{j_l}(s_l)- B_{k_l}(s_l))} \rc,
\end{eqnarray*}
where the inequality follows from the fact that $\be [ e^{ i\sum_{l=1}^m\xi_l  (B_{j_l}(s_l)- B_{k_l}(s_l))}]>0$. Plugging this information into \eqref{eq:dcp-moments-Qm} and taking \eqref{eq:ident-caract-fct} into account, we thus end up with.
\begin{equation*}
\be_{x}\lc\lp  Q_m(t) \rp^{n} \rc
\le
\int_{[0, t]^n}\!\!\int_{(\R^m)^n}
\be_{0}\lc \prod_{l=1}^n \sum_{1\le j<k\le m} e^{i\xi_l  (B_j(s_l)-B_k(s_l))}\rc \, [\mu_{M}^{1}]^{\otimes n}\!(d\xi) \, ds
=
\be_{0}\lc\lp  Q_m(t) \rp^{n} \rc.
\end{equation*}
This proves \eqref{eq:monotone-moments}, and thus our claim.
\end{proof}

We are now ready to prove our main asymptotic theorem for moments of $u$.

\begin{theorem}\label{thm:asymptotic-moments-utx}
Let $t\in\R_{+}$ and $x\in\R$.  
Then the following asymptotic result holds true
\begin{equation}\label{eq:asymptotic-moments-utx}
\lim _{m\to\infty}\frac{1}{m^{1+\frac{1}{H}}} \, \log \lp\be \lc \lp u(t,x)\rp^{m}\rc\rp
 =  \lp\frac{c_{H}}{2}\rp^{\frac{1}{H}} \ce t,
\end{equation}
where $\mathcal{E} =\ce_{1}$ denotes the constant  given by relation \eqref{eq:def-e-theta}.
\end{theorem}

\begin{proof}
Recall that the law of $u(t,x)$ does not depend on $x$, and we thus start from expression~\eqref{eq:fk-with-scaling} for $\be [ (u(t,0))^{m}]$. With a lower bound in mind, our first task will be to relate this expression  to the semi-group $T_{m,t}$ introduced in \eqref{eq:def-semi-group}.

\noindent
\emph{Step 1: Spectral representation.} Let us consider a function $g:\R\to\R$ with support in a given interval $[-a,a]$. For $m\ge 1$ we define $g_{m}=g^{\otimes m}$, that is,
\begin{equation*}
g_{m}: \R^{m}\longrightarrow\R^{m}, \qquad x=(x_{1},\ldots,x_{m})\longmapsto \prod_{j=1}^{m} g(x_{j}).
\end{equation*}
Owing to Lemma \ref{lem:monotone-init-cdt} plus the trivial relation $\frac{g(x)}{\|g\|_{\infty}}\le 1$, we have
\begin{eqnarray*}
\be \lc \lp u(t,0) \rp^{m}\rc
&\ge &    (2a)^{-m}\int_{[-a, a]^m}
\be_{x}\lc\exp\left(  \frac{c_{H} Q_m(t_m)}{m}   \right) \rc  dx\\
&\ge &   \left(2\|g\|_\infty^{2} a \right)^{-m}\int_{[-a, a]^m} g_{m}(x)
\be_{x}\left[ g_{m}( B (t_m))  \exp  \left( \frac{c_{H} Q_m(t_m)}{m}    \right)
\right]dx\,.
\end{eqnarray*}
Invoking our definition \eqref{eq:def-semi-group}, we have thus obtained
\begin{equation}\label{eq:low-bnd-1-utx}
\be \lc \lp u(t,0) \rp^{m}\rc
\ge
\frac{\lla T_{m,t_{m}} g_{m}, \, g_{m}\rra_{L^{2}(\R^{m})}}{\left(2\|g\|_\infty^{2} a \right)^{m}}.
\end{equation}

We now consider a small parameter $\ep>0$. We also recall our definitions \eqref{eq:def-K-theta-m} and \eqref{eq:def-H-m} for $K_{\te,m}$ and $A_{\te,m}$, and set $K_{m}:= K_{c_{H},m}$ (resp. $A_{m}:= A_{c_{H},m}$) to alleviate notations. Owing to our computations in the proof of Proposition \ref{prop1} (Step 5), we can choose $g$ satisfying $\|g\|_{L^{2}(\R)}=1$ and such that
\begin{equation}\label{eq:hyp-g}
\liminf_{m\to\infty} \frac{K_{m}(g_{m})}{m} \ge \ce_{H} -\ep, 
\quad\text{where}\quad
\ce_{H} =\lp\frac{c_{H}}{2}\rp^{\frac{1}{H}} \ce.
\end{equation}
This quantity can be related to the semi-group $T_{m,t}$ in the following way.
Due to the fact that $A_{m}$ is self-adjoint (see Proposition \ref{prop:fk-self-adjoint}), $T_{m,t}$ admits a spectral representation related to its generator. Hence, since $\|g\|_{L^{2}(\R)}=1$, there exists a probability measure $\nu_{g}$ on $\R$ such that:
\begin{equation*}
\langle T_{m,t_{m}} g_{m} , \, g_{m} \rangle_{L^{2}(\R^{m})} 
=
\int_{\R} \exp(t_{m}\la) \, \nu_{g}(d\la)
\ge
\exp\lp t_{m}  \int_{\R} \la \, \nu_{g}(d\la) \rp,
\end{equation*}
which yields
\begin{equation}\label{eq:low-bnd-Ttg-g}
\langle T_{m,t_{m}} g_{m} , \, g_{m} \rangle_{L^{2}(\R^{m})} 
\ge
\exp\lp t_{m}  \langle A_{m} g , \, g\rangle_{L^{2}(\R^{m})} \rp
=
\exp\lp t_{m}  K_{m}(g_{m}) \rp.
\end{equation}
Plugging this information into \eqref{eq:low-bnd-1-utx} and recalling that $g$ satisfies \eqref{eq:hyp-g}, we end up with the following inequality, valid for $m$ large enough
\begin{equation}\label{eq:low-bnd-2-utx}
\be \lc \lp u(t,0) \rp^{m}\rc
\ge
\frac{\exp\lp t_{m}  K_{m}(g_{m}) \rp}{\left(2\|g\|_\infty^{2} a \right)^{m}}
\ge
\frac{\exp\lp m t_{m}  \lp \ce_{H} - \ep \rp \rp}{\left(2\|g\|_\infty^{2} a \right)^{m}}\,. 
\end{equation}

\noindent
\emph{Step 2: Lower bound.}
Starting from \eqref{eq:low-bnd-2-utx}, the desired lower bound is now easily derived: taking into account the fact that $m t_{m}=m^{1+\frac{1}{H}} t$, we get (for $m$ large enough)
\begin{equation*}
\frac{1}{m^{1+\frac{1}{H}}} \, \log \lp\be \lc \lp u(t,0) \rp^{m}\rc\rp
\ge \lp \ce_{H} - \ep \rp t - \frac{\log\lp 2\|g\|_\infty^{2} a \rp}{m^{\frac{1}{H}}}.
\end{equation*}
Taking limits in $m$ and recalling that we have chosen an arbitrarily small $\ep$, we have proved that
\begin{equation*}
\liminf_{m\to\infty}\frac{1}{m^{1+\frac{1}{H}}} \, \log \lp\be \lc \lp u(t,0) \rp^{m}\rc\rp
\ge
\ce_{H} t,
\end{equation*}
which is corresponds to our claim.

\noindent
\emph{Step 3: Cut-off in the Feynman-Kac representation.} 
Let us go back to the Feynman-Kac representation of moments for $u(t,x)$ given by \eqref{y3}, and write the quantity $Q_{m}(t)$ therein as
\begin{equation*}
Q_{m}(t)
=
\sum_{1\le j <k \le m}
\int_0^t  \int_{\R}     e^{ i\xi  (B_j(s)- B_k(s))} \mu(d\xi) ds.
\end{equation*}
In order to get the upper bound part of our theorem, we shall replace the quantity $Q_m(t)$, with its diverging high frequency modes, by the quantity $Q_{m,M}^{1}(t)$ defined in \eqref{eq:def-QmM1}. To this aim, recall that $\ga_{M}^{1},\ga_{M}^{2},\mu_{M}^{1},\mu_{M}^{2}$ are defined in \eqref{eq:def-gaM1-gaM2}--\eqref{eq:def-mu1-mu2}, and notice that we have
\begin{equation}\label{eq:def-Q1-Q2-2}
Q_{m}(t) = Q_{m,M}^{1}(t) + Q_{m,M}^{2}(t),
\quad\text{with}\quad
Q_{m,M}^{l}(t) = \sum_{1\le j <k \le m} \int_0^t   \gamma_{M}^{l}(B_j(s)- B_k(s) )ds.
\end{equation}
Our cut-off procedure is now expressed in the following form. For two conjugate exponents $p,q>1$, H\"older's inequality yields
\begin{equation}\label{eq:dcp-moment-u}
\be \lc \lp u(t,0) \rp^{m}\rc
\le
\be_0^{1/p}\lc \exp\lp  \frac{p \, c_{H} \, Q_{m,M}^{1}(t_{m})}{m} \rp \rc
\be_0^{1/q}\lc \exp\lp  \frac{q \, c_{H} \, Q_{m,M}^{2}(t_{m})}{m} \rp \rc .
\end{equation}
We will prove that our study can be reduced to analysis of the term $Q_{m,M}^{1}$ by showing that for an arbitrary $q\ge 1$
\begin{equation}\label{eq:cutoff-negligible-remainder}
\lim_{M\to\infty}\limsup_{m\to\infty} 
\frac{1}{m^{1+\frac{1}{H}}} \log\lp \be_0^{1/q}\lc \exp\lp  \frac{q \, c_{H} \, Q_{m,M}^{2}(t_{m})}{m} \rp \rc  \rp
\le 0.
\end{equation}

\noindent
\emph{Step 4: Proof of \eqref{eq:cutoff-negligible-remainder}.} Let us generalize somehow our problem, and show that for any $\te>0$ we have
\begin{equation}\label{eq:log-lim-R-theta}
\lim_{M\to\infty}\limsup_{m\to\infty} 
\frac{1}{m^{1+\frac{1}{H}}} \log\lp R_{\te,M,m} \rp \le 0,
\quad\text{with}\quad
R_{\te,M,m} = 
\be_0\lc \exp\lp  \frac{\te \, Q_{m,M}^{2}(t_{m})}{m} \rp \rc.
\end{equation}
Also notice that
\begin{equation*}
Q_{m,M}^{2}(t_{m})
=
\frac{1}{2} \sum_{j=1}^{m} \sum_{k\ne j} \int_{0}^{t_{m}} \int_{\R}
e^{ i\xi  (B_j(s)- B_k(s))} \mu_{M}^{2}(d\xi) ds.
\end{equation*}
A rough bound on $R_{\te,M,m}$ is thus obtained by applying H\"older's inequality, similarly to \cite[Section 3]{ChPh}
\begin{eqnarray}\label{eq:bnd-RMm-1}
R_{\te,M,m} & \le &
\prod_{j=1}^{m} 
\be_0^{1/m}\lc \exp\lp  \frac{\te}{2} \sum_{k\ne j} \int_{0}^{t_{m}} \int_{\R}
e^{ i\xi  (B_j(s)- B_k(s))} \mu_{M}^{2}(d\xi) ds \rp \rc  \notag \\
&= &
\be_0\lc \exp\lp  \frac{\te}{2} \sum_{j=2}^{m} \int_{0}^{t_{m}} \int_{\R}
e^{ i\xi  (B_j(s)- B_1(s))} \mu_{M}^{2}(d\xi) ds \rp \rc.
\end{eqnarray}
We shall now prove that
\begin{equation}\label{eq:bnd-RMm-2}
R_{\te,M,m} \le 
\be_0^{m-1}\lc \exp\lp  \frac{\te}{2}  \int_{0}^{t_{m}} \int_{\R}
e^{ i\xi  B(s)} \mu_{M}^{2}(d\xi) ds \rp \rc.
\end{equation}
Indeed, a simple series expansion for the exponential function  in \eqref{eq:bnd-RMm-1} reveals that $R_{\te,M,m}\le \sum_{n\ge 0} \frac{\te^{n}}{2^{n}n!} A_{n}$, with
\begin{equation*}
A_{n}=
\sum_{j_{1},\ldots,j_{n}=2}^{m} \int_{[0,t_{m}]^{n}} \int_{\R^{n}}
\be_0\lc  e^{ i\xi \sum_{l=1}^{n} (B_{j_{l}}(s)- B_{1}(s))}  \rc \, \prod_{l=1}^{n} \mu_{M}^{2}(d\xi_{l}) \,ds_{l}.
\end{equation*}
In addition, observe that
\begin{equation*}
\be_0\lc  e^{ i\xi \sum_{l=1}^{n} (B_{j_{l}}(s)- B_{1}(s))}  \rc
=
\be_0\lc  e^{ i\xi \sum_{l=1}^{n} B_{j_{l}}(s)}  \rc
\be_0\lc  e^{- i\xi \sum_{l=1}^{n}  B_{1}(s)}  \rc.
\end{equation*}
Owing to the fact that $\be_0[  e^{ i\xi \sum_{l=1}^{n} B_{j_{l}}(s)}  ]\ge 0$ and $\be_0[  e^{- i\xi \sum_{l=1}^{n}  B_{1}(s)}  ]\in(0,1)$, we end up with
\begin{eqnarray*}
A_{n}&\le&
\sum_{j_{1},\ldots,j_{n}=2}^{m} \int_{[0,t_{m}]^{n}} \int_{\R^{n}}
\be_0\lc  e^{ i\xi \sum_{l=1}^{n} B_{j_{l}}(s)}  \rc \, \prod_{l=1}^{n} \mu_{M}^{2}(d\xi_{l}) \,ds_{l}  \\
&=&
\be_0\lc \lln \sum_{j=2}^{m}  \int_{0}^{t_{m}} \int_{\R}
e^{ i\xi  B_j(s)} \mu_{M}^{2}(d\xi) \, ds \rrn^{n}
\rc.
\end{eqnarray*}
Invoking the series expansion of the exponential function again, this last bound easily yields our claim~\eqref{eq:bnd-RMm-2}.

Starting from \eqref{eq:bnd-RMm-2}, we can prove \eqref{eq:cutoff-negligible-remainder}. To this aim, a direct consequence of \eqref{eq:bnd-RMm-2} is the following inequality
\begin{equation}\label{eq:bnd-RMm-3}
\frac{1}{m^{1+\frac{1}{H}}} \log\lp R_{\te,M,m} \rp
\le 
\frac{1}{m^{\frac{1}{H}}} \log\lp \be_0\lc \exp\lp  \frac{\te}{2}  \int_{0}^{t_{m}} \int_{\R}
e^{ i\xi  B(s)} \mu_{M}^{2}(d\xi) ds \rp \rc \rp.
\end{equation}
Furthermore, some positivity arguments similar to the one we resorted to in Lemma \ref{lem:QmM-increase-M} show that
\begin{equation*}
\be_0\lc \exp\lp  \frac{\te}{2}  \int_{0}^{t_{m}} \int_{\R} e^{ i\xi  B(s)} \mu_{M}^{2}(d\xi) ds \rp \rc
\le
\be_0\lc \exp\lp  \frac{\te}{2}  \int_{0}^{[t_{m}]+1} \int_{\R} e^{ i\xi  B(s)} \mu_{M}^{2}(d\xi) ds \rp \rc,
\end{equation*}
where $[x]$ denotes the integer part of a real number $x$. Hence, applying successively Markov's property for $B$ and Lemma \ref{lem:monotone-init-cdt}, we get 
\begin{eqnarray*}
\be_0\lc \exp\lp  \frac{\te}{2}  \int_{0}^{t_{m}} \int_{\R} e^{ i\xi  B(s)} \mu_{M}^{2}(d\xi) ds \rp \rc
&\le&
\sup_{x\in\R} \be_{x}^{[t_{m}]+1}\lc 
\exp\lp  \frac{\te}{2}  \int_{0}^{1} \int_{\R} e^{ i\xi  B(s)} \mu_{M}^{2}(d\xi) ds \rp \rc  \\
&\le&
\be_{0}^{[t_{m}]+1}\lc 
\exp\lp  \frac{\te}{2}  \int_{0}^{1} \int_{\R} e^{ i\xi  B(s)} \mu_{M}^{2}(d\xi) ds \rp \rc .
\end{eqnarray*}
Plugging this inequality into  \eqref{eq:bnd-RMm-3} yields
\begin{eqnarray*}
\frac{1}{m^{1+\frac{1}{H}}} \log\lp R_{\te,M,m} \rp
&\le&
\frac{[t_{m}]+1}{m^{1+\frac{1}{H}}}
\log\lp \be_0\lc 
\exp\lp  \frac{\te}{2}  \int_{0}^{1} \int_{\R} e^{ i\xi  B(s)} \mu_{M}^{2}(d\xi) ds \rp \rc \rp  \\
&\le&
2 t \,
\log\lp \be_0\lc 
\exp\lp  \frac{\te}{2}  \int_{0}^{1} \int_{\R} e^{ i\xi  B(s)} \mu_{M}^{2}(d\xi) ds \rp \rc \rp .
\end{eqnarray*}
Summarizing, we have obtained that
\begin{equation*}
\lim_{M\to\infty}\limsup_{m\to\infty} 
\frac{1}{m^{1+\frac{1}{H}}} \log\lp R_{\te,M,m} \rp
\le
\lim_{M\to\infty} 
\log\lp \be_0\lc 
\exp\lp  \frac{\te}{2}  \int_{0}^{1} \int_{\R} e^{ i\xi  B(s)} \mu_{M}^{2}(d\xi) ds \rp \rc \rp = 0,
\end{equation*}
where the last relation can be seen by means of some elementary computations, very similar to the ones displayed in \cite[p. 50]{HHLNT}. This finishes the proof of \eqref{eq:cutoff-negligible-remainder}.

\noindent
\emph{Step 5: An expression with diagonal terms.}
Let us now focus on the behavior of the quantity $Q_{m,M}^{1}(t)$ defined by \eqref{eq:def-Q1-Q2-2}. Specifically, having \eqref{eq:def-Q1-Q2-2} and \eqref{eq:dcp-moment-u} in mind, we wish to find the asymptotic behavior of
\begin{equation*}
\ca_{m,M}^{1}(\te) \equiv
\be_0\lc  \exp\lp \frac{\te Q_{m,M}^{1}(t_{m})}{m} \rp \rc ,
\end{equation*}
for a given parameter $\te>0$.
We wish to reduce this asymptotic study to an evaluation involving Feynman-Kac semigroups. A first step in this direction is to realize that, since the cut-off measure $\mu_{M}^{1}$ is now finite, the asymptotic behavior of $m\mapsto\be[(u(t,x))^{m}]$ is not perturbed by adding the diagonal terms corresponding to $j=k$ in the sum defining $Q_{m,M}^{1}(t)$. That is, one can replace $Q_{m,M}^{1}(t)$ by
\begin{equation}\label{eq:def-hat-Q-1}
\hat{Q}_{m,M}^{1}(t) =  \sum_{j,k=1}^{m}
\int_0^t  \int_{\R}     e^{ i\xi  (B_j(s)- B_k(s))} \mu_{M}^{1}(d\xi) ds
=
\int_0^t  \int_{\R} \lln h_{B}(s,\xi)\rrn^{2} \mu_{M}^{1}(d\xi) ds,
\end{equation}
where we have set $h_{B}(s,\xi)= \sum_{j=1}^{m} e^{ i\xi  B_j(s)}$.

Indeed, let us recast Definition \eqref{eq:def-QmM1} as 
\begin{equation*}
Q_{m,M}^{1}(t) = \sum_{1\le j <k \le m} q_{m,M}^{jk}(t),
\quad\text{where}\quad
q_{m,M}^{jk}(t) = \int_0^t  \int_{\R}     e^{ i\xi  (B_j(s)- B_k(s))} \mu_{M}^{1}(d\xi) ds.
\end{equation*}
Now it is readily checked that the diagonal terms $q_{m,M}^{jj}$ can be expressed as
\begin{equation*}
q_{m,M}^{jj}(t)
=
\iot \int_{-M}^{M} |\xi|^{1-2H} d\xi \, ds
=
c_{1,H} t M^{2-2H}.
\end{equation*}
Thus, recalling the definition \eqref{eq:def-hat-Q-1} of $\hat{Q}_{m,M}^{1}(t)=\sum_{j,k=1}^{m} q_{m,M}^{jk}(t)$, we have
\begin{equation}\label{eq:def-hat-cal-AmM}
\hat{\ca}_{m,M}^{1}(\te)
:=
 \be_{0}\lc\exp\lp\frac{\te \, \hat{Q}_{m,M}^{1}(t_{m})}{m}\rp\rc
=
\exp\lp \te  t   M^{2-2H} m^{\frac{1}{H}} \rp
\ca_{m,M}^{1}(2\te),
\end{equation}
where we have replaced $\te$ by $2\te$ as a parameter of $\ca_{m,M}^{1}$, due to repetitions of off diagonal terms in the definition of $\hat{Q}_{m,M}^{1}$.
This easily entails that
\begin{equation}\label{eq:equal-AmM-hat-AMm}
\limsup_{m\to\infty} 
m^{-\frac{1+H}{H}}\log \lp \ca_{m,M}^{1}(\te) \rp
=
\limsup_{m\to\infty} 
m^{-\frac{1+H}{H}}\log \lp \hat{\ca}_{m,M}^{1} \rp.
\end{equation}
Recalling once again our formula \eqref{y3}, we now focus on the evaluation of $\hat{\ca}_{m,M}^{1}$.

\noindent
\emph{Step 6: A coarse graining procedure.}
Thanks to relation \eqref{eq:equal-AmM-hat-AMm}, our problem is reduced to a Feynman-Kac asymptotics for the semi-group related to $\hat{\ca}_{m,M}^{1}$. However, the semi-group $\hat{T}_{m,M,t}$ is considered in Proposition \ref{prop:fk-self-adjoint-truncated} with $m,M$ fixed and $t\to\infty$. In contrast, we consider here a situation where both $m$ and $t_{m}$ are going to $\infty$. We solve this problem by a coarse graining type procedure which is described below.

To this aim, consider a fixed $\rho\in\N$ and let us decompose $m$ as $m=n\rho +r$ with $n,r\in\N$ and $0\le r \le n-1$. We can write
\begin{equation*}
\hat{Q}_{m,M}^{1}(t)) = 
\int_{0}^{t}  \int_{\R} \lln \sum_{j=0}^{n} \sum_{k=0}^{\rho_{j}-1} 
e^{ i\xi  B_{j\rho+k}(s)} \rrn^{2} \mu_{M}^{1}(d\xi) ds,
\end{equation*}
where $\rho_{j}=\rho$ for $0\le j\le n-1$ and $\rho_{n}=r+1$. Then invoke the elementary relation $|\sum_{j=0}^{n} z_{j}|^{2} \le (n+1) \sum_{j=0}^{n} |z_{j}|^{2}$, valid on $\C$, to get
\begin{equation*}
\hat{Q}_{m,M}^{1}(t)) \le (n+1)  \sum_{j=0}^{n}
\int_{0}^{t}  \int_{\R} \lln  \sum_{k=0}^{\rho_{j}-1} 
e^{ i\xi  B_{j\rho+k}(s)} \rrn^{2} \mu_{M}^{1}(d\xi) ds\,. 
\end{equation*}
Thanks to the independence of the Brownian motions and the definition \eqref{eq:def-hat-cal-AmM} of $\hat{\ca}_{m,M}^{1}$, we thus obtain
\begin{equation}\label{eq:dcp-I-R}
\hat{\ca}_{m,M}^{1}
=\be_{0}\lc\exp\lp\frac{\te \hat{Q}_{m,M}^{1}(t_{m})}{m}\rp\rc  
\le
\lln \hat{I}_{\rho,M}(t_{m}) \rrn^{n} \, \hat{R}_{m,M}(t_{m}),
\end{equation}
where we set $\te_{\rho}=\frac{\te (n+1)}{m}$ and
\begin{equation*}
\hat{I}_{\rho,M}(t_{m})
=
\be_0\lc \exp\lp \te_{\rho} 
\int_{0}^{t_{m}}  \int_{\R} \lln  \sum_{k=1}^{\rho}  e^{ i\xi  B_{k}(s)} \rrn^{2} \mu_{M}^{1}(d\xi) ds \rp  \rc .
\end{equation*}
In relation \eqref{eq:dcp-I-R}, the remainder term $\hat{R}_{m,M}(t_{m})$ is defined by
\begin{equation*}
\hat{R}_{m,M}(t_{m})
=
\be_0\lc \exp\lp \te_{\rho}
\int_{0}^{t_{m}}  \int_{\R} \lln  \sum_{k=1}^{r}  e^{ i\xi  B_{k}(s)} \rrn^{2} \mu_{M}^{1}(d\xi) ds \rp  \rc,
\end{equation*}
and notice that we also have
\begin{equation}\label{eq:lim-theta-rho}
\lln  \te_{\rho}- \frac{\te}{\rho} \rrn
\le \frac{\te}{\rho \, n} .
\end{equation}

Let us now evaluate the term $\hat{I}_{\rho,M}(t_{m})$ above. To this aim, we linearize our expression again and write
\begin{equation*}
\int_{0}^{t}  \int_{\R} \lln  \sum_{k=1}^{\rho}  e^{ i\xi  B_{k}(s)} \rrn^{2} \mu_{M}^{1}(d\xi) ds
=
\sum_{j,k=1}^{\rho} \int_{0}^{t} \ga_{M}^{1}\lp B_{j}(s) - B_{k}(s) \rp \, ds.
\end{equation*}
Related to this expression, consider the semi-group $\htt_{\rho,M,t}$ defined by \eqref{eq:def-semi-group-trunc}, for a parameter $\te=\te_{\rho}$. Then it is readily checked that $\hat{I}_{\rho,M}(t_{m})=\htt_{\rho,M,t}\1(0)$. We can now apply Proposition~\ref{prop:fk-self-adjoint-truncated} and relation \eqref{eq:lim-theta-rho}, which yields
\begin{equation}\label{eq:lim-InM}
\lim_{m\to\infty} \frac{\log\lp \hat{I}_{\rho,M}(t_{m})\rp}{t_{m}} 
=
\la_{\rho,M},
\end{equation}
where $\la_{\rho,M}$ is given by \eqref{eq:def-la-mM} with $\te$ replaced by $\frac{\te}{\rho}$. Along the same lines, one can also check that
\begin{equation}\label{eq:lim-RnM}
\lim_{m\to\infty} \frac{\log\lp \hat{R}_{m,M}(t_{m})\rp}{t_{m}} = 0.
\end{equation}

\noindent
\emph{Step 6: Conclusion for the upper bound.}
Let us go back to decomposition \eqref{eq:dcp-I-R} and invoke relation \eqref{eq:lim-RnM} in order to write
\begin{eqnarray*}
\frac{1}{m^{1+\frac{1}{H}}} \log\lp \hat{\ca}_{m,M}^{1}(\te) \rp
&\le&
\frac{n}{m^{1+\frac{1}{H}}} \log\lp \hat{I}_{\rho,M}(t_{m}) \rp + \ep_{m,M} \\
&=& 
\lp \frac{t}{\rho} \rp \lp\frac{m-r}{m} \rp  \frac{1}{t_{m}} \ \log\lp \hat{I}_{\rho,M}(t_{m}) \rp + \ep_{m,M},
\end{eqnarray*}
where $\lim_{m\to\infty}\ep_{m,M} =0$. Plugging the limiting behavior \eqref{eq:lim-InM} into this inequality, we end up with
\begin{equation*}
\limsup_{m\to\infty} \frac{1}{m^{1+\frac{1}{H}}} \log\lp \hat{\ca}_{m,M}^{1}(\te) \rp
\le
t \, \frac{\la_{\rho,M}}{\rho}.
\end{equation*}
This bound is valid for any $\rho\in\N$, and a small variation of Proposition \ref{prop1} (replacing the distribution $\ga$ by its smoothed version $\ga_{M}^{1}$ and summing over diagonal terms) asserts that for all $M > 0$ there exists a quantity $\la_{M}(\te)$ verifying:
\begin{equation*}
\lim_{\rho\to\infty} \frac{\la_{\rho,M}}{\rho} = \la_{M}(\te),
\quad\text{and}\quad
\lim_{M\to\infty}  = \la_{M}(\te) = \ce_{\te}.
\end{equation*}
This yields our upper bound taking into account \eqref{eq:cutoff-negligible-remainder}, the fact that we consider $\ca_{m,M}^{1}(\te)$ with $\te=\frac{p c_{H}}{2}$ and $p$ arbitrarily close to 1, plus relation \eqref{eq:equal-AmM-hat-AMm}.
\end{proof}

Using Proposition  Corollary 1.2.4 in \cite{Ch1} we deduce the following corollary.

\begin{corollary}\label{c.2.8}
Consider the solution $u$ of equation \eqref{eq:spde-pam}. Then the following tail estimate holds true
\begin{equation}\label{eq:tail-ut0}
\lim_{a\to\infty}a^{-(1+H) }\log\lp \PP\lp \log( u(t,x))\ge   a\rp \rp
=- \hat{c}_{H,t} ,
\end{equation}
where the constant $\hat{c}_{H,t}$ is defined by
\begin{equation*}
\hat{c}_{H,t} = \lc \frac{(1+H) \, c_{H}}{2} \lp \lp 1+ \frac{1}{H} \rp t \ce  \rp^{H} \rc^{-1}.
\end{equation*}
\end{corollary}

\begin{remark}\label{rmk-cHt-qHt}
Denote by $q_{H,t}$ the quantity showing up in the right hand side of equation \eqref{eq:as-limit-in space}. Then we have $\hat{c}_{H,t} = q_{H,t}^{-(1+H)}$.
\end{remark}

\begin{proof}[Proof of Corollary \ref{c.2.8}]
Recall that we have proved Theorem \ref{thm:asymptotic-moments-utx}. Furthermore, some simple arguments based on H\"older's inequality allow us  to extend this result to real valued powers. Namely, the limit in \eqref{eq:asymptotic-moments-utx} can be taken along positive real numbers instead of integers.

Let us thus consider a sequence of real numbers $(a_{n})_{n\ge 1}$ converging to $\infty$. Related to this sequence, we also introduce the sequence of random variables $Y_{n}=a_{n}^{-1/H}\log(u(t,x))$, and the sequence $(\rho_{n})_{n\ge 1}$ with $\rho_{n}=a_{n}^{1+\frac{1}{H}}$. Then a direct application of \eqref{eq:asymptotic-moments-utx} yields, for an additional parameter $\beta\ge 0$
\begin{equation}\label{eq:lim-ellis-gartner}
\lim_{n\to\infty} \frac{1}{\rho_{n}} \log\lp \be\lc \exp\lp \beta \rho_{n} Y_{n} \rp \rc  \rp
= \laa(\beta),
\quad\text{where}\quad
\laa(\beta)\equiv \lp\frac{c_{H}}{2}\rp^{\frac{1}{H}} t \ce \beta^{1+\frac{1}{H}} .
\end{equation}
Notice that the above asymptotic result does not enable a direct application of Ellis-Gartner's theorem, since the limit in \eqref{eq:lim-ellis-gartner} is only obtained for $\beta\ge 0$ (while Ellis-Gartner would require limits for $\beta<0$ too). We will thus apply a large deviation theorem for positive random variables (Theorem 1.2.3 in \cite{Ch1}), which can be summarized as follows. Assume relation~\eqref{eq:lim-ellis-gartner} holds true for $\beta\ge 0$, and  define  $\laa^{*}$ as
\begin{equation}\label{eq:fenchel-adjoint}
\laa^{*}(\la)= \sup\lcl   \la \beta - \laa(\beta) \rcl.
\end{equation}
If $\laa^{*}$ is smooth and convex on $(0,\infty)$, and if $\lim_{\la\to\infty}\laa^{*}(\la)=\infty$, then the following tail estimate holds true for $\la>0$
\begin{equation}\label{eq:tail-ldp}
\lim_{n\to\infty} \frac{1}{\rho_{n}} \log\lp\bp\lp Y_{n} \ge \la  \rp\rp
= -\laa^{*}(\la).
\end{equation}
The application of the latter result raises 2 additional questions that we address now, similarly to what is done in \cite{Ch2}:

\noindent
\emph{(i)}
Our random variable $Y_{n}$ is not positive. However, observe that
\begin{equation*}
\text{sgn}(Y_{n}) = \text{sgn}(Y_{1}) = \text{sgn}\lp \log(u(t,x)) \rp,
\quad\text{and}\quad
\bp\lp Y_{1}>0 \rp >0.
\end{equation*}
We thus decompose the exponential moments of $Y_{n}$ into
\begin{align*}
&\lefteqn{\be\lc \exp\lp \beta \rho_{n} Y_{n} \rp \rc} \\
&=
\be\lc \exp\lp \beta \rho_{n} Y_{n} \rp \big| \, Y_{1} < 0 \rc \, \bp\lp Y_{1} < 0 \rp
+
\be\lc \exp\lp \beta \rho_{n} Y_{n} \rp \big| \, Y_{1} \ge 0 \rc \, \bp\lp Y_{1} \ge 0 \rp \\
&\le
1+ \be\lc \exp\lp \beta \rho_{n} Y_{n} \rp \big| \, Y_{1} \ge 0 \rc.
\end{align*}
Owing to this relation, plus our limiting result \eqref{eq:lim-ellis-gartner}, it is readily checked that
\begin{equation*}
\lim_{n\to\infty} \frac{1}{\rho_{n}} \log\lp \be\lc \exp\lp \beta \rho_{n} Y_{n} \rp \big| \, Y_{1} \ge 0 \rc  \rp
= \laa(\beta).
\end{equation*}
Furthermore, conditioned to the event $(Y_{1} \ge 0)$, the random variable $Y_{n}$ can be considered as positive.

\noindent
\emph{(ii)}
One can easily compute $\laa^{*}$ thanks to relation \eqref{eq:fenchel-adjoint}, and the reader can check that $\laa^{*}(\la) = \hat{c}_{H,t} \la^{1+H}$. In particular, $\laa^{*}$ is smooth and convex on $(0,\infty)$, and we also have $\lim_{\la\to\infty}\laa^{*}(\la)=\infty$.

Taking into account the last considerations, we have thus obtained a conditioned version of \eqref{eq:tail-ldp}, that is
\begin{equation*}
\lim_{n\to\infty} \frac{1}{\rho_{n}} \log\lp\bp\lp Y_{n} \ge \la \big| \, Y_{1} \ge 0  \rp\rp
= -\laa^{*}(\la).
\end{equation*}
One can then transform this relation into a nonconditioned one, owing to the fact that $\bp(Y_{1} \ge 0)$ is a fixed strictly positive quantity. Recalling the notation for $Y_{n}$ and $\rho_{n}$, we end up with
\begin{equation*}
\lim_{n\to\infty} \frac{1}{a_{n}^{1+\frac{1}{H}}} \log\lp\bp\lp \log( u(t,x)) \ge a_{n}^{\frac{1}{H}}\la  \rp\rp
= -\hat{c}_{H,t} \la^{1+H}.
\end{equation*}
Some elementary changes of variable now yield our claim \eqref{eq:tail-ut0}.
\end{proof}

\begin{remark}
Notice that the constant $\hat{c}_{H,t}$ appearing in (\ref{eq:tail-ut0}) is precisely $c_0(H)^{-\frac 1{1+H}} (t\mathcal{E})^{-H}$, where $c_0(H)$ is defined in (\ref{c0}).\end{remark}

\section{Proof of Theorem  \ref{thm:space-asymptotics}}\label{sec:proof}

In this section we start from the tail behavior for the random variable $\log( u(t,x))$ provided by \eqref{eq:tail-ut0}. We carry out a localization and discretization procedure which will allow us to evaluate the growth of $x\mapsto u(t,x)$. 

\subsection{Proof of the lower bound by localization}

Our approach is based on a method (introduced in \cite{CJK} and \cite{CJKS}) involving localizations of the driving noise $\dot{W}$ in the space-time domain. Let us start by some elementary preliminaries (whose proofs are left to the reader) concerning the localizing function.

\begin{lemma}
Let $\ell$ be the function defined by:
\[
\ell (x)=\frac{1-\cos x}{\pi x^2}\,,\ \ x\in \RR\,,
\quad\text{or equivalently}\quad
\cf  \ell (\xi)=(1-|\xi|)I_{\left\{ |\xi|\le 1\right\}}\,, \quad\xi \in \RR\,.
\]
For any $\beta>0$ define $\ell_\beta(x)=\beta \ell(\beta x)$.
Then the Fourier transform of $\ell_\beta(x)$ is given by:
\begin{equation} \label{b1}
\cf  \ell _\beta (\xi)=\left(1-\frac{|\xi|}{\beta} \right)I_{\left\{ |\xi|\le \beta\right\}}\,, \quad\xi \in \RR\,.
\end{equation}
\end{lemma}

We now give a representation of the noise $\dot{W}$ as a convolution of a certain kernel with respect to a space-time white noise. 

\begin{lemma}
Let $\tilde{\gamma}$ be the distribution defined by
\begin{equation*}
\tilde{\gamma}(x)= \lp \frac{c_H}{2\pi}\rp^{1/2}  \int_{\R} e^{i\xi x } |\xi| ^{\frac 12-H} d\xi
=\lp  \frac {c_H} {2\pi} \rp^{1/2} \cf\lp |\xi| ^{\frac 12-H}\rp.
\end{equation*}
Then the Gaussian field $\{W(t, \phi), t\ge 0 , \phi \in \mathcal{S}(\R)\}$  introduced in  (\ref{notation}) can be represented as
\begin{equation}\label{eq:W-representation-white-noise}
W(t, \phi) =\int_0^t \int_\RR ( \phi*\tilde{\gamma})(x) \hat{W}(ds,dx),
\end{equation}
where $\hat{W}$ is a standard space-time white noise on $\R^{2}$.
\end{lemma}

\begin{proof}
It is easily checked that
$\tilde{\gamma} *\tilde{\gamma}= c_H\gamma$ 
is the spatial covariance of the fractional Brownian sheet $W$. In fact,
\[
\tilde{\gamma} *\tilde{\gamma}
= \frac {c_H} { 2\pi}  \mathcal{F}  ( |\xi| ^{\frac 12-H} )* \mathcal{F}  ( |\xi| ^{\frac 12-H} )
= c_H\mathcal{F} \mu =c_H \gamma.
\]
Let now $W$ be the Gaussian field given by \eqref{eq:W-representation-white-noise}.
Then for any $s,t \ge 0$ and $ \phi, \psi \in \mathcal{S}(\R)$, we can write:
\begin{eqnarray}
\EE [ W(t,\phi) W(s, \psi)]&=&(s\wedge t) \langle \phi *\tilde{\gamma}, \psi *\tilde{\gamma}\rangle_{L^2(\R)} 
= (2\pi)^{-1} (s\wedge t)\langle \mathcal{F}(\phi*\tilde{\gamma}), \mathcal{F}(\psi*\tilde{\gamma}) \rangle_{L^2(\R)} \nonumber\\
 &=& (2\pi) ^{-1} (s\wedge t) \langle \mathcal{F}\phi \mathcal{F}\tilde{\gamma}, \mathcal{F}\psi\mathcal{F}\tilde{\gamma} \rangle_{L^2(\R)} \nonumber \\  \label{b2}
&=&c_H (s\wedge t) \int_{\R} \mathcal{F}\phi(\xi) \overline{ \mathcal{F}\psi(\xi)} \mu(d\xi),
\end{eqnarray}
which corresponds to expression \eqref{eq:cov1}.
\end{proof}

We now turn to a description of the localized approximation of $u$ which will be used in the sequel. For this step, we fix a parameter $\beta\ge 1$ and consider  the  approximation $\left\{W_\beta(t, \phi)\,, \phi\in \cs (\RR)\right\}$ of the fractional Brownian field  $W$  defined by
\begin{equation}  \label{b1a}
W_\beta(t,\phi)
=\int_0^t\int_\RR(  [(\mathcal{F}\ell_\beta)\tilde{\gamma}]*\phi)(x)      \hat{W}(ds,dx).
\end{equation}  
From (\ref{b1a}) we obtain the following expression for the covariance function of the random field $\{W_{\beta}(t, \phi), t\ge 0 , \phi \in \mathcal{S}(\R)\}$:
\begin{equation}
\EE\left[W_\beta(t,\phi)W_\beta(t,\psi)\right]
= tc_H  \int_\RR   (\ell_\beta * |\cdot|^{\frac 12-H})^2 (\xi) \mathcal{F}\phi(\xi) \overline{ \mathcal{F}\psi(\xi)} d\xi\,.
\end{equation}
Notice that $\ell_\beta$ is an approximation of the identity as $\beta$ tends to infinity, so that $W_{\beta}$ has to be seen as an approximation of $W$. On the other hand, the  spatial covariance  of the noise $W_\beta$, given by  $(\mathcal{F}\ell_\beta)\tilde{\gamma} *(\mathcal{F}\ell_\beta)\tilde{\gamma}$, has compact support. In this sense we call it \emph{localized}.

\begin{remark}
We have followed the notation of \cite{CJK} for our localization step. However let us stress the fact that, though the localization is made through the Fourier transform of $\ell_{\beta}$, it is a localization in \emph{direct} spatial coordinates.
\end{remark}

Having the approximation \eqref{b1a} in hand, we can now define the following Picard approximation of the solution $u$ to equation \eqref{eq:spde-pam}. Namely, we set $U_{\beta, 0}(t,x)=1$ and for $n\ge 1$, we define
\begin{equation}\label{eq:recursion-Un}
U_{\beta, n+1}(t,x)=1 +\int_0^t \int_{x-\beta  \sqrt t}^{x+\beta \sqrt t}
p_{t-s}(y-x) U_{\beta, n}(s,y) W_\beta(ds,dy).
\end{equation}
The next result states the independence of the random variables  $ \{ U_{\beta, n}(t,x_i), i\ge 1 \}$ if the points  $x_i$'s are far enough from each other. This property is crucial in order to establish our almost sure spatial behavior.

\begin{lemma}\label{indep} 
Choose any fixed $\beta\ge 1$ and let $n=[\log \beta]+1$.  Then for any sequence of points
$\{ x_1<x_2< \cdots\}$   such that $x_{i+1}-x_i>2n \beta (1+\sqrt t)$ for each  $i\ge 1$, 
the random variables $\{U_{\beta, n}(t,x_i)  , i\ge 1\} $, defined by \eqref{eq:recursion-Un},
are independent.  
\end{lemma}

\begin{proof} The proof is exactly the same as the proof of Lemma 5.4 of \cite{CJKS}, and is omitted for sake of conciseness.
\end{proof}

Next let us recall the following elementary result borrowed from \cite{HHLNT}, which will help us to bound moments of iterated integrals.
\begin{lemma}\label{lem:intg-simplex}
For $m\ge 1$ let $\alpha \in (-1+\ep, 1)^m$  with $\ep>0$ and  set $|\alpha |= \sum_{i=1}^m
\alpha_i  $. For $t\in\ott$, the $m$-th  dimensional simplex over $\ot$ is denoted by
$S_m(t)=\{(r_1,r_2,\dots,r_m) \in \R^m: 0<r_1  <\cdots < r_m < t\}$.
Then there is a constant $c>0$ such that
\[
J_m(t, \alpha):=\int_{S_m(t)}\prod_{i=1}^m (r_i-r_{i-1})^{\alpha_i}
dr \le \frac { c^m t^{|\alpha|+m } }{ \Gamma(|\alpha|+m +1)},
\]
where by convention, $r_0 =0$.
\end{lemma}

With these preliminaries in hand, we can now proceed to a first approximation result relating $u$, $U_{\beta}$ and $U_{\beta, n}$.

\begin{proposition}\label{thm:exist-uniq-chaos}
 For any $(t,x) \in \R_+\times \R$ and for any $p\ge 1$, the sequence    $U_{\beta, n}(t,x)$ defined by \eqref{eq:recursion-Un} converges in $L^p$ to a random variable
$U_{\beta}(t,x)$ as $n$ tends to infinity. Furthermore, we have the following estimates for the differences  of the solutions:   there is
a finite constant $C$ (dependent on $t$ but independent of $\beta$ and $p$) such that for any $\beta\ge 1$ and $p\ge 2$, 
\begin{equation}
\| u(t,x)-U_{\beta}(t,x) \|_{L^{p}(\oom)}
\le \beta^{-(\frac  12-H)}\exp\left( Cp^{\frac1{ H}}   \right) 
\label{e.beta-difference} 
\end{equation} 
and for two constants $c_{1},c_{2}>0$:
\begin{equation}
\| U_{\beta, n}(t,x) -U_{\beta}(t,x) \|_{L^{p}(\oom)}
\le  \frac{(c_{1} p^{1/2})^{n}}   {\Gamma\lp\frac{nH}{2} +1\rp} \, \exp\lp c_{2} p^{\frac{1}{H}}\rp\,. 
\label{e.n-difference} 
\end{equation} 
\end{proposition}

\begin{proof} The solution $u(t,x)$ admits a chaos expansion
(see, for instance,  formula (5.9) in \cite{HHLNT}).  Namely, we have 
\begin{equation}\label{eq:chaos-expansion-u(tx)}
u(t,x)=\sum_{n=0}^{\infty}I_n(f_n(\cdot,t,x))\,,
\end{equation}
where $f_0(t,x)=1$ and for any $n\ge 1$, 
\begin{equation}\label{eq:expression-fn}
f_n(s_1,x_1,\dots,s_n,x_n,t,x)
=\frac{1}{n!}p_{t-s_{\si(n)}}(x-x_{\si(n)})\cdots p_{s_{\si(2)}-s_{\si(1)}}(x_{\si(2)}-x_{\si(1)})\,.
\end{equation}
Here $\si$ denotes the permutation of $\{1,2,\dots,n\}$ such that $0<s_{\si(1)}<\cdots<s_{\si(n)}<t$
and $I_n$ is the multiple It\^o-Wiener integral with respect to the fractional Brownian field $W$.

The same kind of formula holds for $U_{\beta,n}$. Namely, denote by $p^{(\beta)}$ the kernel defined by
\[
p^{(\beta)}(s,t;  y)=p _{t-s}(y) \1_{\{|y|\le \beta \sqrt t\}}\,.
\]
Then for $n\ge 1$, one can recast formula \eqref{eq:recursion-Un} as
\[
U_{\beta, n+1}(t,x)=1+\int_0^t \int_\RR
p^{(\beta)} (s,t; y-x) U_{\beta, n}(s,y) W_\beta(ds,dy).
\]
By iteration, similarly to \cite{HHLNT} and \eqref{eq:chaos-expansion-u(tx)}, we have
\begin{equation}
\label{eq:chaos-expansion-u(txbetan)}
U_{\beta, n}(t,x) =\sum_{k=0}^{n}I_{\beta,k}(f_{\beta, k} (\cdot,t,x))\,,
\end{equation}
where $f_{\beta,0}(t,x) =1$ and for $k\ge 1$,
\begin{equation}\label{eq:expression-fn-2}
f_{\beta, k} (s_1,x_1,\dots,s_k,x_k,t,x)\\
=\frac{1}{k!}p^{(\beta)}(s_{\si(k)} ,t; x-x_{\si(k)})\cdots p^{(\beta)}(s_{\si(1)}, s_{\si(2)}; x_{\si(2)}-x_{\si(1)}).
\end{equation}
In the above expression, $I_{\beta, k}$ is the multiple It\^o-Wiener integral of order $k$ with respect to the Gaussian process
$W_\beta(t,x)$.  
We are going to show that the sequence $U_{\beta, n}(t,x)$ converges in $L^2$, and  defines a random field
$U_{\beta}(t,x)   $ given by 
\begin{equation}
\label{b4}
U_{\beta  }(t,x) =\sum_{n=0}^{\infty}I_{\beta,n}(f_{\beta, n} (\cdot,t,x)).
\end{equation}
On the other hand, we will also see that $U_{\beta  }(t,x) $  converges in $L^2$ to $u(t,x)$ as $\beta$ tends to infinity.
 In order to simplify the notation we will omit the dependence on $(t,x)$ in some of the terms of our computations.  The proof will be done in several steps.
 
\medskip
\noindent{\it Step 1: Estimates on a fixed chaos.} 
In this step we consider a fixed chaos $n$. Having expressions \eqref{eq:chaos-expansion-u(tx)} and\eqref{b4} in mind, we shall estimate the expectation $\EE[|   I_{n}(f_ n) -I_{\beta,n}(f_{\beta, n}) |^2]$. Let us start with the following decomposition:
\begin{eqnarray}\label{eq:def-A1-A2}
\EE[|   I_{n}(f_ n) -I_{\beta,n}(f_{\beta, n}) |^2]
&\le&  2\EE[| I_{n}(f_ n)-  I_{n}(f_{\beta, n})  |^2]+ 2 \EE[|   I_{n}(f_{\beta, n})-I_{\beta,n}(f_{\beta, n})|^2]  \notag \\
&=:&  2(A_1+A_2),  
\end{eqnarray}
and let us estimate $A_{1}$ and $A_{2}$ separately.

Consider  first the term $A_1$. Some elementary computation reveals that the Fourier transform of $f_n$ is given by
\[
\cf f_n(s_1,\xi_1,\dots,s_n,\xi_n)=
\frac{1}{n!}  { e^{-ix (\xi_{\sigma(n)}+ \cdots + \xi_{\sigma(1)})}}
\prod_{i=1}^n e^{-\frac{1}{2}(s_{\si(i+1)}-s_{\si(i)})|\xi_{\si(i)}+\cdots +
\xi_{\si(1)} |^2}   ,
\]
where we  have used the convention $s_{\si(n+1)}=t$. The same kind of computation can be performed for $f_{\beta,n}$. Specifically, let $\rho^{(\beta)} (s,t; \xi)$ be the Fourier transform of $p^{(\beta)}(s,t;x)$, namely,
\begin{equation}\label{eq:def-rho-beta}
\rho^{(\beta)} (s,t; \xi)=\frac{1}{\sqrt{2\pi(t-s)}}\int_{-\beta \sqrt t}^{\beta \sqrt t} e^{-ix\xi -\frac{x^2}{2(t-s)}} dx\,.
\end{equation}
Then the Fourier transform of $f_{\beta, n}$ is given  by
\begin{equation}\label{k1}
\cf f_{\beta, n} (s_1,\xi_1,\dots,s_n,\xi_n) =
\frac{1 }{n!} 
\prod_{i=1}^n\rho^{(\beta)}(s_{\si(i)}, s_{\si(i+1)};  \xi_{\si(i)}+\cdots +
\xi_{\si(1)} )  
{ e^{-ix (\xi_{\sigma(n)}+ \cdots + \xi_{\sigma(1)})}}.
\end{equation}
In the sequel, we also make use of the notation
\begin{equation}\label{eq:def-tilde-rho-j}
\tilde{\rho}^{(\beta)}_j  =    \rho^{(\beta)}(s_{\si(j)}, s_{\si(j+1)};  \xi_{\si(j)}+\cdots +
\xi_{\si(1)} ) -  e^{-\frac{1}{2}(s_{\si(j+1)}-s_{\si(ji)})|\xi_{\si(j)}+\cdots +
\xi_{\si(1)} |^2}\,.
\end{equation}
Now a straightforward application of Parseval's identity yields
\begin{align*}
A_1&= n! \|\cf f_n-\cf f_{\beta, n}\|^2_ {  L^2([0,t]^n\times \R^n, \lambda^n \times \mu^{n})}  \\
& =  \frac 1{ n!}   \int_{[0,t]^n}\int_{\RR^n}\bigg|
{ e^{-ix (\xi_{\sigma(n)}+ \cdots + \xi_{\sigma(1)})}}
\prod_{i=1}^n     \rho^{(\beta)}(s_{\si(i)}, s_{\si(i+1)};  \xi_{\si(i)}+\cdots +\xi_{\si(1)} )  
   \\
  & \hspace{1in}-  
{ e^{-ix (\xi_{\sigma(n)}+ \cdots + \xi_{\sigma(1)})}}
\prod_{i=1}^n e^{-\frac{1}{2}(s_{\si(i+1)}-s_{\si(i)})|\xi_{\si(i)}+\cdots +\xi_{\si(1)} |^2} 
\bigg|^2  \prod_{i=1}^n  |\xi_i |^{1-2H}   d\xi ds,
\end{align*}
where we have set $d\xi=d\xi_1 \cdots d\xi_n$,  $ds=ds_1\cdots ds_n$ and where $\lambda$ denotes the Lebesgue measure. 
Thus, using a telescoping sum argument,  we can write
\begin{equation}  \label{h9}
 A_1 \le  \frac n{ n!}   \sum_{j=1}^n A_{j,n},
\end{equation}
where we recall that we have defined $\tilde{\rho}_{j}^{(\beta)}$ in \eqref{eq:def-tilde-rho-j}, and where we set:
\begin{multline}\label{j1a} 
A_{j,n} =
 \int_{[0,t]^n}\int_{\RR^n} \left| \tilde{\rho}^{(\beta)}_j   \right|^2
    \prod_{i=j+1}^{n} e^{-\frac {1 }{2} (s_{\si(i+1)}-s_{\si(i)})|\xi_{\sigma(i)}+\cdots +\xi_{\sigma(1)} |^2}\\
\times \prod_{i=1}^{j-1}   \bigg|  \rho^{(\beta)}(s_{\si(i)}, s_{\si(i+1)};  \xi_{\si(i)}+\cdots +
\xi_{\si(1)} ) \bigg|^2 
  \prod_{i=1}^n  |\xi_i |^{1-2H} d\xi ds\,.
\end{multline}
For the time being,    denote $t=s_{\si(j+1)}$, $s=s_{\si(j)}$, $\eta=\xi_{\sigma(j)}+ \cdots + \xi_{\sigma(1)}$, let $C$ be a generic constant (possibly depending on $H$ and $t$), and let us estimate $\tilde{\rho}_j^{(\beta)}$. It is easy to see that
\begin{eqnarray*}
&& \rho^{(\beta)}(s,t;\eta)
-e^{-\frac12 (t-s)\eta^2}
=\frac{1}{\sqrt{2\pi (t-s)}}  \int_{|x|\ge \beta \sqrt t} e^{ix\eta-
\frac{x^2}{2(t-s)}}dx \\
&=& \frac{1}{\sqrt{2\pi (t-s)}}  \int_{|x|\ge \beta \sqrt t} \cos(\eta x) e^{ -
\frac{x^2}{2(t-s)}}dx \\
&=& \frac{-2}{\eta \sqrt{2\pi (t-s)}}    e^{ -
\frac{\beta^2t }{2(t-s)}}\sin (\eta \beta \sqrt t)-
\frac{1}{\eta \sqrt{2\pi (t-s)}}  \int_{|x|\ge \beta \sqrt t} \sin (\eta x) d e^{ -
\frac{x^2}{2(t-s)}} \,.
\end{eqnarray*}
Therefore, trivially bounding $|\sin (\eta x)|$ by 1 in the integral above, we get
\[
|\rho^{(\beta)}(s,t;\eta)
-e^{-\frac12 (t-s)\eta^2}|\le
\frac{C }{|\eta| \sqrt{ t-s }} e^{-\frac {\beta^2}{2}}.
\]
On the other hand, it is readily checked from \eqref{eq:def-tilde-rho-j} that $|\rho^{(\beta)}(s,t;\eta)|$  is  bounded by a constant $C$.
Thus, for any $\theta\in [0, 1]$ (possibly depending on $\eta$), we have
\begin{equation} \label{b7}
|\rho^{(\beta)}(s,t;\eta)
-e^{-\frac12 (t-s)\eta^2}|\le \frac{C}{|\eta|^\theta  ( t-s)^{\theta/2} }  e^{-\frac {\beta^2\theta }{2}}.
\end{equation}
Substituting this bound into \eqref{j1a} yields
\begin{multline}\label{b9}
A_{j,n}\le  C^2  e^{-\beta^2\theta }\int_{[0,t]^n}\int_{\RR^n}
\frac{(s_{\si(j+1)}-s_{\si(j)})^{- \theta }}{|\xi_{\sigma(j)}+\cdots +\xi_{\sigma(1)} |^{2 \theta } }  
 \prod_{i=j+1}^{n}
e^{- (s_{\si(i+1)}-s_{\si(i)})|\xi_{\sigma(i)}+\cdots +\xi_{\sigma(1)} |^2} \\ 
\times   \prod_{i=1}^{j-1}    \left| \rho^{(\beta)}(s_{\si(i)}, s_{\si(i+1)};  \xi_{\si(i)}+\cdots +
\xi_{\si(1)} ) \right|^2 
  \prod_{i=1}^n  |\xi_i |^{1-2H} d\xi ds.  
\end{multline}
Making the change of variable $\xi_{\sigma(i)}+\cdots + \xi_{\sigma(1)}=\eta
_{i}$, for all $i=1,2,\dots, n$, we obtain
\begin{multline*}
A _{j,n} 
\le C^2  e^{-\beta^2\theta }\ \int_{[0, t]^n}\int_{\RR^n} 
(s_{\si(j+1)}-s_{\si(j)})^{-\theta}     |\eta_j |^{-2\theta}   
\prod_{i=j+1}^{n}
e^{- (s_{\si(i+1)}-s_{\si(i)}) \eta_i^2 } \\
 \times \prod_{i=1}^{j-1}  |\rho^{(\beta)}(s_{\si(i)},s_{\si(i+1)};
\eta_i)|^2    \prod_{i=1}^n  |\eta_i-\eta_{i-1} |^{1-2H}  d\eta ds,
\end{multline*}
where we have set $\eta_{0}=0$.   
 We can now invoke the elementary bound $|\eta_i-\eta_{i-1} |^{1-2H} \le |\eta_i |^{1-2H} + | \eta_{i-1} |^{1-2H} $, and
we obtain
\begin{multline}\label{b10}
A_{j,n} 
\le C^{2}  e^{-\beta^2\theta }\sum_{\al\in D_{n}} \int_{[0, t]^n}\int_{\RR^n}
(s_{\si(j+1)}-s_{\si(j)} )^{-\theta}
 |\eta_j |^{-2\theta}
\prod_{i=j+1}^{n} e^{-   (s_{\si(i+1)}-s_{\si(i)})\eta_i   ^2 }  \\
\times  \prod_{i=1}^{j-1}    |\rho^{(\beta)}(s_{\si(i)},s_{\si(i+1)};
\eta_i)|^2   \prod_{i=1}^n  |\eta_i|^{\al_i}  d\eta ds \,,
\end{multline}   
where  $D_{n}$ is a subset of multi-indices of length $n$ satisfying the following rules: $\text{Card}(D_{n})=2^{n}$ and for any $\al\in D_{n}$ we have
\begin{equation}\label{b11}
|\al|\equiv \sum_{i=1}^{n} \alpha_i = n(1-2H),
\quad\text{and}\quad
\al_{i} \in \{0, 1-2H, 2(1-2H)\}, \quad i=1,\ldots, n.
\end{equation}

Now we perform the integration on each variable $\eta_i$ in relation \eqref{b10}. 
If $i\ge j+1$,  it is readily checked that
\begin{equation} \label{f1}
\int_\RR |\eta_i|^{\al_i}  e^{-   (s_{\si(i+1)}-s_{\si(i)})\eta_i   ^2 }d\eta_i=C (s_{\si(i+1)}-s_{\si(i)}) ^{-\frac{\al_i+1}{2}}\,.
\end{equation}
In the case  $i\le j-1$,   we also claim that
\begin{equation} \label{h13}
 \int_\RR   |\rho^{(\beta)}(s_{\si(i)},s_{\si(i+1)};
\eta_i)|^2  |\eta_i|^{\alpha_i} d\eta_i\le C  (s_{\si(i+1)}-s_{\si(i)} )^{-\frac{\al_i+1}{2}}.
\end{equation}
In fact, recalling our definition \eqref{eq:def-rho-beta} and thanks to an elementary change of variable, this integral can be written as
\[
\frac 1{2\pi}  (s_{\si(i+1)}-s_{\si(i)} )^{-\frac{\al_i+1}{2}}
 \int_\R |\xi|^{\alpha_i}   \left| \int_{-\beta \sqrt{s_{\si(i+1)}} /\sqrt{s_{ \si(i+1)} -s_{\si(i)}}}
 ^ {\beta \sqrt{s_{\si(i+1)}} /\sqrt{ s_{\si(i+1)}- s_{\si(i)}}} e^{ix \xi -\frac {x^2} 2} dx \right|^2 d\xi,
 \]
 which is bounded by a constant times  $(s_{\si(i+1)}-s_{\si(i)} )^{-\frac{\al_i+1}{2}}$ by Lemma \ref{lemA} in the Appendix below.
It remains to consider the integral over the variable $\eta_j$ in \eqref{b10}, which is given by
\begin{equation*}
(s_{\si(j+1)}-s_{\si(j)} )^{-\theta} 
 \int_{\RR} |\eta_j |^{-2\theta +\al_j} 
 d\eta_j \,.
\end{equation*}
 We decompose the integral $\int_\R |\eta_j|^{-2\theta +\alpha_j} d\eta_j$ into two parts: on the region $|\eta_j| \le1$ we take $\theta= \frac {\alpha_j+1} 2 - \frac \delta 2$ and on the region  $|\eta_j| >1$ we take $\theta= \frac {\alpha_j+1} 2 + \frac \delta 2$,  for some $\delta >0$ to be fixed later.   In this way,  we obtain 
\begin{align} \label{e.intergalj-1-j}
&(s_{\si(j+1)}-s_{\si(j)} )^{-\theta} 
 \int_{\RR} |\eta_j |^{-2\theta +\al_j} 
 d\eta_j  \notag\\
 &\le
 C\lp  (s_{\si(j+1)}-s_{\si(j)} )^{-\frac12(\al_{j}+1-\delta)} 
 + (s_{\si(j+1)}-s_{\si(j)} )^{-\frac12(\al_{j}+1+\delta)} \rp \notag \\
& \le
 C \, (s_{\si(j+1)}-s_{\si(j)} )^{-\frac12(\al_{j}+1+\delta)}.
\end{align}
Therefore, plugging \eqref{f1}, \eqref{h13} and \eqref{e.intergalj-1-j} into \eqref{b10}, we end up with the following relation for any $1\le j\le n$
\begin{eqnarray}
A_{j,n}  
&\le&  
C e^{-\beta^2\theta }  \sum_{\alpha}  \int_{[0,t]^n}  \prod_{i\not = j}
(s_{\si(i+1)}-s_{\si(i)} )^{-\frac{\al_i+1}{2}}  \label{h15}
(s_{\si(j+1)}-s_{\si(j)} )^{-\frac {\alpha_j+1}2 - \frac \delta 2 }ds  \notag\\
&=&
C e^{-\beta^2\theta } n! \sum_{\alpha}  \int_{S_{n}(t)}  
\prod_{i\not = j} (s_{i+1}-s_{i} )^{-\frac{\al_i+1}{2}}  \label{h15}
(s_{j+1}-s_{j} )^{-\frac {\alpha_j+1}2 - \frac \delta 2 }ds,
\end{eqnarray}
where $\theta = \frac {\alpha_j +1}2 -\frac \delta 2 $.
We now wish to apply Lemma \ref{lem:intg-simplex} in order to bound the right hand side of \eqref{h15}, and we first discuss the nature of the exponents involved: (i)~First we have to ensure that each exponent in the integral showing up in \eqref{h15} is lower bounded by $-1$. These exponents are $\ge -\frac {\max_{i\le n}\alpha_i+1}{2} -\frac{\delta}{2}$, and recall that $\max_{i\le n}\alpha_i\le 2(1-2H)$ according to \eqref{b11}. We can thus ensure that each exponent is greater than $-1$, provided that $0< \delta  < 4H-1$. (ii)~Invoking relation \eqref{b11} again, it is readily checked that the sum of the exponents in  (\ref{h15}) is  $-n+nH -\frac \delta 2$. (iii)~We wish to choose $\theta$ as large as possible in order to ensure the maximal exponential decay for $A_{j,n} $. According to our previous considerations, we have taken  $\theta= \frac {\alpha_j+1} 2 - \frac \delta 2$ with $\delta$ of the form $4H-1-\ep$ for an arbitrarily small $\ep$. Referring once more to (\ref{h15}), we can just ensure $\al_{j}\ge 0$, which yields
$\te\ge   \frac {1-\delta}2  = 1-2H +\frac \ep 2 \ge 1-2H$. 
With those considerations in mind, we  can now apply Lemma \ref{lem:intg-simplex} to relation \eqref{h15} in order to conclude that
\begin{equation} \label{h3}
  A_{j,n}
\leq   \frac{   n! C^n  }{  \Gamma(nH+\frac 12)} \, e^{-\beta^2(1-2H)},
\end{equation}
where $C$ is a constant depending on $H$ and $ t$. Substituting (\ref{h3}) into (\ref{h9}) yields (recall that the constant $C$ might change from line to line)
\[
A_1 \le   \frac{    C^n   }{  \Gamma(nH+\frac 12)} \, e^{-\beta^2(1-2H)}.
\]
Using the inequality  $\Gamma(nH+\frac 12) \ge \frac {\Gamma(nH+1)}{nH+\frac 12}$, we obtain
\begin{equation} \label{k2}
A_1 \le   \frac{    C^n   }{  \Gamma(nH+1)} \, e^{-\beta^2(1-2H)}.
\end{equation}

Going back to our decomposition \eqref{eq:def-A1-A2}, let us now deal with the term $A_2$.  The spectral measure of the noise $W_\beta$ has a density equal to  $c_H(\ell_\beta * |\cdot|^{\frac 12-H})^2 (\xi)$. Therefore, thanks to another telescoping sum argument, we get
\begin{eqnarray*}
A_2& = & c_H^n n!   \int_{[0,t]^n}\int_{\RR^n} 
  |\cf f_{\beta, n} (s_1,\xi_1,\dots,s_n,\xi_n)|^2 \left| \prod_{i=1}^n |\xi_i|^{\frac 12-H} -  \prod_{i=1}^n(\ell_\beta * |\cdot|^{\frac 12-H} (\xi_i) \right|^2 d\xi ds \\
  &\le &  c_H^n  n n! \sum_{j=1}^n  \int_{[0,t]^n}\int_{\RR^n} 
  |\cf f_{\beta, n} (s_1,\xi_1,\dots,s_n,\xi_n)|^2  \prod_{i=1}^{j-1} |\xi_i|^{1-2H}  \\
  &&\hspace{1.5in}\times 
  \left|  |\xi_j|^{\frac 12-H}  -(\ell_\beta * |\cdot|^{\frac 12-H}) (\xi_j) \right|^2 
    \prod_{i=j+1}^n (\ell_\beta * |\cdot|^{\frac 12-H} )^2(\xi_i)
   d\xi ds.
 \end{eqnarray*}
 In addition, notice that  
 \[
 | |\xi_j|^{\frac 12-H}  -(\ell_\beta * |\cdot|^{\frac 12-H}) (\xi_j)|
 \le c_{1,H}\beta ^{-(\frac12-H)},
 \]
 where $c_{1,H}= \int_{\RR} |\eta|^{\frac 12-H} \ell(\eta) d\eta$. This follows easily from
 \[
  | |\xi_j|^{\frac 12-H}  -(\ell_\beta * |\cdot|^{\frac 12-H}) (\xi_j)|
  \le \beta^{ -(\frac 12-H)} \int_{\R} \ell(\eta)  \left| |\beta \xi _j|^{\frac 12-H} - |\beta \xi _j-\eta|^{\frac 12-H} \right|d\eta,
  \]
  and the inequality  $ \left| |\beta \xi_j |^{\frac 12-H} - |\beta \xi _j-\eta|^{\frac 12-H} \right| \le |\eta| ^{\frac 12-H}$.
  In the same way we can show that
 \[
  (\ell_\beta * |\cdot|^{\frac 12-H} )(\xi_i)
  \le |\xi_i|^{\frac 12-H} + c_{1,H} \beta^{-(\frac12-H)}.
  \]
Taking into account that $\beta \ge 1$,  this leads to the estimate
 \[
 A_2 \le C^n n!  \beta^{-(1-2H)} \int_{[0,t]^n}\int_{\RR^n} 
  |\cf f_{\beta, n} (s_1,\xi_1,\dots,s_n,\xi_n)|^2 \prod_{i=1}^n (|\xi_i|^{1-2H} \vee 1) d\xi ds.
  \]
We now start from the expression \eqref{k1} for $\cf f_{\beta, n} (s_1,\xi_1,\dots,s_n,\xi_n)$, we make the change of variable $\xi_{\sigma(i)}+\cdots + \xi_{\sigma(1)}=\eta_{i}$, for all $i=1,2,\dots, n$, and we bound $|\eta_i-\eta_{i-1} |^{1-2H}$ by $|\eta_i |^{1-2H} + | \eta_{i-1} |^{1-2H}$ as in the case of our term $A_{1}$. This yields
\begin{eqnarray}\label{k3}   \nonumber
 A_2 &\le& C^n   \beta^{-(1-2H)} \int_{S_{n}(t)}\int_{\RR^n} 
  |\rho^{(\beta)}(s_i, s_{i+1)};   \eta_i )|^2 \prod_{i=1}^n (|\eta_i-\eta_{i-1}|^{1-2H} \vee 1) d\eta ds\\ \nonumber
  &\le& C^n   \beta^{-(1-2H)} \sum_{\al\in D_{n}}  \int_{S_{n}(t)}\int_{\RR^n} 
  |\rho^{(\beta)}(s_i, s_{i+1});   \eta_i )|^2 \prod_{i=1}^n (|\eta_i|^{\alpha_i} \vee 1) d\eta ds \\
  &\le &  C^n   \beta^{-(1-2H)}  \sum_{\al\in D_{n}} \int_{S_{n}(t)} (s_{i+1}- s_i) ^{-\frac {\alpha_i+1}2} ds 
  \le       \frac{C^n}{ \Gamma(nH+1)} \, \beta^{-(1-2H)},
\end{eqnarray}
where we recall that $S_{n}(t)$ denotes the $n$-dimensional simplex of $[0,t]^{n}$.
We now conclude, putting together (\ref{k2}) and (\ref{k3}), that
\begin{equation} \label{h10}
\EE[|   I_{n}(f_ n) -I_{\beta,n}(f_{\beta, n}) |^2]
\le  \frac {C^n} { \Gamma(nH+1 )} \, \beta^{-(1-2H)}.
\end{equation}
In a similar way, we can also obtain the following estimate (whose proof is left to the patient reader), where the constant $C$ is independent of $\beta$
\begin{equation}
\EE[|   I_{\beta,n}(f_{\beta, n}) |^2]\le  \frac{     C^n     }{ \Gamma(nH+1)}  \,.
\label{h11}
\end{equation} 

\noindent
{\it Step 2: $L^{p}$-estimates.} 
Recall that $U_{\beta,n}(t,x)$ is defined by the finite sum \eqref{eq:chaos-expansion-u(txbetan)}.
Let us first get the convergence of this finite sum  to a random variable $U_{\beta}(t,x)$ formally defined by the series \eqref{b4}. To this aim, recall that for a functional $F_{n}$ which belongs to the $n$-th chaos of a Wiener space and $p\ge 2$, we have the hypercontractivity inequality  $\|F_{n}\|_{L^{p}(\oom)}\le p^{\frac{n}{2}} \|F_{n}\|_{L^{2}(\oom)}$. We thus get
\begin{eqnarray*}
\| U_{\beta, n}(t,x) -U_{\beta}(t,x) \|_{L^{p}(\oom)}
&\le&
\sum_{k=n+1}^{\infty} \| I_{\beta,k}(f_{\beta, k}) \|_{L^{p}(\oom)}
\le
\sum_{k=n+1}^{\infty}  p^{\frac{n}{2}} \| I_{\beta,k}(f_{\beta, k}) \|_{L^{2}(\oom)}  \\
&\le&
\sum_{k=n+1}^{\infty}  \frac{(C p^{1/2})^{n}}  { \Gamma\lp\frac{nH}{2}+1\rp},
\end{eqnarray*}
where the last inequality is due to \eqref{h11}. Furthermore, the following inequality, valid for $z\ge 0$ and $a>0$, is an easy consequence of estimates on Mittag-Leffler functions which can be found in \cite{EMOF}
\begin{equation*}
\sum_{k=n+1}^{\infty} \frac{z^{n}}{\gga(ak+1)} 
\le \frac{c_{1} z^{n}}{\gga(an+1)} \, e^{c_{2} z^{\frac{1}{a}}},
\end{equation*}
where $c_{1},c_{2}$ are two universal constants. Plugging this bound into our previous estimate, we end up with
\begin{equation*}
\| U_{\beta, n}(t,x) -U_{\beta}(t,x) \|_{L^{p}(\oom)}
\le
\frac{(c_{3} p^{1/2})^{n}} {\Gamma\lp\frac{nH}{2}+1 \rp} \, \exp\lp c_{4} p^{\frac{1}{H}}\rp,
\end{equation*}
which shows our claim \eqref{e.n-difference}.

The same kind of consideration also allows to derive inequality \eqref{e.beta-difference}. Namely, write
\begin{equation*}
\| u(t,x) -U_{\beta}(t,x) \|_{L^{p}(\oom)}
\le
\sum_{k=0}^{\infty} \| I_{n}(f_ n) -I_{\beta,n}(f_{\beta, n}) \|_{L^{p}(\oom)}
\le
\beta^{-(1-2H)} \, \sum_{k=0}^{\infty} \frac{(C p^{1/2})^{n}} { \Gamma\lp\frac{nH}{2}+1\rp},
\end{equation*}
where we resort to hypercontractivity and \eqref{h10} for the last step.
This easily yields \eqref{e.beta-difference} by the same kind of argument as before. 
\end{proof}

\begin{corollary}\label{cor:cvgce-Un-beta-u}
Consider $p\ge 1$.
Under the same assumptions as in Proposition  \ref{thm:exist-uniq-chaos},  suppose that $\beta =\exp(Mp^{\frac 1 H})$ and
$n= [\log \beta]+1$ for some constant $M>0$.  Then,  for any $\nu >0$, there exists $M>0$, such that
 \begin{equation}  \label{equ1}
 \| u(t,x) -U_{\beta,n}(t,x) \|_{L^{p}(\oom)} \le \exp \left\{  - \nu p^{\frac 1H} \right\}.
 \end{equation}
\end{corollary}

\begin{proof}
From (\ref{e.beta-difference}) and (\ref{e.n-difference}) we obtain
\[
\| u(t,x) -U_{\beta,n}(t,x) \|_{L^{p}(\oom)} 
\le 
e^{C p^{\frac 1H}}  \left ( \beta^{-(\frac12-H)} + \frac {C^n p^{n/2}} {\Gamma(\frac {nH} 2+\frac 12)} \right),
 \]
 for some constant $C$ depending on $H$ and $t$. Using the asymptotic properties of the Gamma function, this is
 bounded by
 \begin{equation}\label{h12}
 e^{C p^{\frac 1H}}  \left ( \beta^{-(\frac12-H)} +  C^n p^{\frac  n2} n^{-\frac {nH} 2} \right).
 \end{equation}
To bound the above right-hand side,  we have first (recall that $\beta =\exp(Mp^{\frac 1 H})$)
\[
 e^{C p^{\frac 1H}}  \beta^{-(\frac12-H)} = \exp\left\{ p^{\frac 1H} \lc C- \lp \frac 12-H\rp M \rc \right\},
 \]
 which is less than   $\frac 12 \exp(  - \nu p^{\frac 1H } )$ for $M$ large enough. For the second summand in the right-hand side of \eqref{h12}, we obtain (provided $n= [\log \beta]+1$) the upper bound
 \[
 \exp\left\{ p^{\frac 1H} \lp C + M \log C -\frac M2 \log M\rp  + \frac 12 \log p \right\},
 \]
  which again  is less than   $\frac 12 \exp(  - \nu p^{\frac 1H })$ for $M$ large enough. This completes the proof of the corollary.
\end{proof}

We are now ready to give the proof of our lower bound.

\begin{proposition}\label{prop:lower-bound}
Under the assumptions of Theorem \ref{thm:space-asymptotics}, for all $t>0$ we have
\begin{equation}\label{eq:spatial-lower-bound}
\liminf_{R\to\infty}(\log R)^{-{\frac{1}{1+H}}}\log\lp \max_{\vert x\vert\le R}u(t,x) \rp
\ge \lc \hat{c}_{H,t}\rc^{-\frac{1}{1+H}} \  \hskip.2in \text{a.s.},
\end{equation}
where $\hat{c}_{H,t}$ is defined in Corollary \ref{c.2.8} and is related to \eqref{eq:as-limit-in space} by Remark \ref{rmk-cHt-qHt}.
\end{proposition}

\begin{proof}
We divide this proof in two steps: first we determine a main contribution to the maximum, given by our approximations $U_{\beta,n}$ suitably discretized. Then we will evaluate the main contribution.

\noindent
\textit{Step 1: Fluctuation results.}
Fix $R>0$ and consider a given $\nu>0$.
Referring to the notation of Corollary \ref{cor:cvgce-Un-beta-u}, we wish to choose $p$ in inequality \eqref{equ1} such that we obtain
\begin{equation}\label{l0}
\EE\lc  |u(t,x) -U_{\beta,n}(t,x)|^{p} \rc
\le
R^{-\nu}.
\end{equation}
It is readily checked that this is achieved for $p=p(R)=(\log R) ^{\frac {H}{1+H}}$, which is greater than~1 if $R\ge e$. We thus choose this $p$, the corresponding $\beta$ and $n$ in Corollary \ref{cor:cvgce-Un-beta-u} being then given by
\begin{equation}\label{l1}
\beta=\beta(R)=\exp\Big\{M(\log R)^{\frac{1}{1+H}}\Big\}
\quad\text{and}\quad
n = \lc M(\log(R))^{\frac{1}{1+H}} \rc + 1.
\end{equation}

For a fixed $t>0$, we now wish to produce some independent random variables $U_{\beta,n}(t,x_{j})$, with  $x_{j}\in[-R,R]$. 
For this we need   $x_{j+1}-x_j>2n \beta (1+t^{1/2})$ for all $j$. Set $N=2n\beta (1+   t^{1/2})$. We choose the set of points
\[
\cn_R=\left\{ 2kN:  k\in \mathbb{Z}, - \left[ \frac R{2N}\right] \le k \le \left[ \frac R{2N}\right]  \right\}.
\]
If $|\cn_{R}| = \left( 2  \left[ \frac R{2N}\right] +1 \right)$ denotes  the cardinality of $\cn_{R}$, one can check, using the expressions of $\beta$ and $n$ given in (\ref{l1}),   that for any $\ep>0$    the following inequalities hold
\begin{equation}\label{l2}
c_{H,M,\ep} \, R^{1-\ep} \le |\cn_{R}| \le R,
\end{equation}
for $R$ large enough, 
where $c_{H,M,\ep}$ is a positive constant.

We can now study $\max_{z\in { \mathcal N}_R}\vert u(t,z)-U_{\beta, n} (t,z)\vert$. For any $\eta>0$ we have 
\begin{multline}\label{l3}
\PP\Big(\log\max_{z\in { \mathcal N}_R}\big\vert u(t,z)
-U_{\beta, n} (t,z)\big\vert\ge
\eta (\log R)^{\frac{1}{1+H}}\Big)  \\
\le
|\cn_{R}| \, 
\PP\Big(\log\big\vert u(t,0)-U_{\beta,n } (t,0)\big\vert\ge
\eta (\log R)^{\frac{1}{1+H}}\Big)
\end{multline}
Furthermore, a simple application of Markov's inequality yields, for an arbitrary $p\ge 1$
\begin{equation*}
\PP\Big(\log\big\vert u(t,0)-U_{\beta,n } (t,0)\big\vert\ge
\eta (\log R)^{\frac{1}{1+H}}\Big)
\le
\frac{\EE\lc  |u(t,x) -U_{\beta,n}(t,x)|^{p} \rc}{\exp\lp \eta p  (\log R)^{\frac{1}{1+H}} \rp},
\end{equation*}
so that choosing $p=p(R)$ and invoking relation \eqref{l0}, we can recast this relation as
\begin{equation}\label{l3.5}
\PP\Big(\log\big\vert u(t,0)-U_{\beta,n } (t,0)\big\vert\ge
\eta (\log R)^{\frac{1}{1+H}}\Big)
\le R^{-(\nu+\eta)}.
\end{equation}
Going back to inequality \eqref{l3} and choosing $\nu=3$, we have obtained the following inequality for $R$ large enough
\begin{equation}\label{l4}
\PP\Big(\log\max_{z\in { \mathcal N}_R}\big\vert u(t,z)
-U_{\beta, n} (t,z)\big\vert\ge
\eta (\log R)^{\frac{1}{1+H}}\Big)  
\le
R^{-2}.
\end{equation}
Notice that this decay in $R$ is sufficient to apply Borel-Cantelli's lemma. Considering for instance a sequence $R=m$, we get
\begin{equation}
\lim_{m\to\infty}(\log m )^{-{\frac{1}{1+H}}}
\log\max_{z\in { \mathcal N}_{m}}\big\vert u(t,z)
-U_{\beta(m),n(m)}(t,z)\big\vert=0\hskip.2in \text{a.s},
\label{e.2.21}
\end{equation}
which is enough to assert that
\begin{equation}\label{l5}
\liminf_{R\to\infty}(\log R)^{-{\frac{1}{1+H}}}\log\lp \max_{\vert x\vert\le R}u(t,x) \rp
\ge
\liminf_{R\to\infty}(\log R)^{-{\frac{1}{1+H}}}\log\lp \max_{x\in\cn_{R}}U_{\beta,n}(t,x) \rp,
\end{equation}
where we recall that $\beta=\beta(R)$ and $n=n(R)$ in the right-hand side of \eqref{l5} are given by~\eqref{l1}. We will now evaluate the right-hand side of \eqref{l5}, identified with our main contribution.

\noindent
\textit{Step 2: Evaluation of the main term.}
Fix $\lambda>0$ and $\delta>0$ arbitrarily small, satisfying the following condition
\begin{equation}
\hat{c}_{H,t} \lp \lambda+\delta\rp^{1+H} < 1 - \delta,
\label{e.2.22}
\end{equation}
where $\hat{c}_{H,t}$ is the constant introduced in Corollary \ref{c.2.8} (observe that $\lambda$ is arbitrarily close to $\hat{c}_{H,t}^{-1/(1+H)}$).
Using the independence property established in  Lemma  \ref{indep}, we can write
\begin{multline}\label{equ2}
\PP\Big(\log\max_{z\in { \mathcal N}_R} \vert U_{\beta,n}(t,z)\vert
\le\lambda(\log R)^{\frac{1}{1+H}}\Big)  \\
= \bigg(1-\PP\Big(\log \vert U_{\beta,n}(t,0)\vert
>\lambda(\log R)^{\frac{1}{1+H}}\Big)\bigg)^{|\cn_{R}|}.
\end{multline}
In the following lines, we will write $u=u(t,0)$ and $U_{\beta,n}=U_{\beta,n}(t,0)$ to alleviate notations.
Then recall that for $a>0$ and $b\in \R$, the following elementary relation holds true
\begin{equation}\label{l5.5}
\log a \le \log 2+ \max\lcl \log(|a-b|), \, \log (|b|) \rcl.
\end{equation}
Applying this inequality to $a=u$ and $b=|U_{\beta,n} |$, we obtain
\[
\log  u \le
\log 2+\max \left\{ 
\log \lp | u  -U_{\beta,n}    |\rp, \,  \log ( | U_{\beta,n} | ) \right\}.
\]
Hence, if we assume that $R$ is large enough, so that $\log 2 \le \delta (\log R) ^{\frac 1{1+H}}$, we have
\begin{align*}
&\lcl \log u  > (\lambda+\delta )(\log R)^{\frac{1}{1+H}} \rcl
\subset
\lcl  \max \left\{ 
\log \lp | u  -U_{\beta,n}    |\rp, \,  \log \lp |U_{\beta,n} |\rp  \right\}
> \la (\log R)^{\frac{1}{1+H}}
 \rcl  \\
 &\hspace{1.3in}=
 \lcl  \log \lp | u  -U_{\beta,n}    |\rp > \la (\log R)^{\frac{1}{1+H}} \rcl
 \cup
  \lcl  \log \lp | U_{\beta,n}|\rp > \la (\log R)^{\frac{1}{1+H}} \rcl
\end{align*}
Owing to simple additivity properties of $\PP$, we thus get
\begin{align}\label{l6}
&\PP\Big(\log ( |U_{\beta,n}|) >\lambda(\log R)^{\frac{1}{1+H}}\Big) \notag \\
&\ge
\PP\Big(\log  u> (\lambda+\delta )(\log R)^{\frac{1}{1+H}}\Big)
-
\PP\Big(\log \lp | u  -U_{\beta,n}    |\rp> \lambda(\log R)^{\frac{1}{1+H}}\Big) \notag\\
&\ge
\PP\Big(\log  u> (\lambda+\delta )(\log R)^{\frac{1}{1+H}}\Big)
- \frac{1}{R^{3+\la}},
\end{align}
where the last inequality is a direct consequence of \eqref{l3.5}. Now recall that we have chosen $\la$ fulfilling condition \eqref{e.2.22}. Applying Corollary \ref{c.2.8} in this context yields
\begin{equation*}
\lim_{R\to\infty} \frac{1}{\log R}  \log \PP\Big(\log   u> (\lambda+\delta )(\log R)^{\frac{1}{1+H}}\Big)
>1-\delta,
\end{equation*}
and thus, for $R$ large enough the following holds true
\begin{equation*}
\PP\Big(\log  u> (\lambda+\delta )(\log R)^{\frac{1}{1+H}}\Big) > \frac{1}{R^{1- \delta }}.
\end{equation*}
Plugging this relation into \eqref{l6} gives, for $R$ large enough
\begin{equation}\label{l7}
\PP\Big(\log ( |U_{\beta,n}|) >\lambda(\log R)^{\frac{1}{1+H}}\Big) > \frac{1}{R^{1-\frac{\delta}{2}}}
\end{equation}
We now gather \eqref{equ2}, \eqref{l2} and \eqref{l7} in order to get
\begin{equation*}
\PP\Big(\log\max_{z\in { \mathcal N}_R} \vert U_{\beta,n}(t,z)\vert
\le\lambda(\log R)^{\frac{1}{1+H}}\Big)
\le
\lp  1- \frac{1}{R^{1-\frac{\delta}{2}}} \rp^{|\cn_{R}|}
\le
\exp\lp -c_{H,M,\ep} R^{\frac{\delta}{2}-\ep}\rp.
\end{equation*}
In conclusion, since we can choose $\ep<\frac{\delta}{2}$,  we have established the bound
$$
\PP\Big\{\log\max_{z\in { \mathcal N}_R} \vert U_{\beta,n} (t,z)\vert
\le\lambda(\log R)^{\frac{1}{1+H}}\Big\}
\le \exp\big\{-R^v\big\}
$$
for some  $v>0$ and for $R$ large enough. Resorting again to Borel-Cantelli's lemma, this implies
\begin{equation}
\liminf_{m\to\infty}(\log m)^{-{\frac{1}{1+H}}}
\log\max_{z\in { \mathcal N}_{m} }\vert U_{\beta(m),n}(t,z)
\vert\ge\lambda\hskip.2in \text{a.s.}\label{e.2.23}
\end{equation}

\noindent
\textit{Step 3: Conclusion.}
Combining the above inequality \eqref{e.2.23} with \eqref{e.2.21}, we have obtained:
$$
\liminf_{m\to\infty}(\log m)^{-{\frac{1}{1+H}}}
\log\max_{z\in { \mathcal N}_{m}} u(t,z)\ge\lambda\hskip.2in \text{a.s.}
$$
By the fact that
$$
\max_{z\in { \mathcal N}_R} u(t,z)\le\max_{\vert x\vert\le R} u(t,x)
$$
and by the monotonicity of $\displaystyle\max_{\vert x\vert\le R} u(t,x)$ in $R$, we can now easily deduce that:
$$
\liminf_{R\to\infty}(\log R)^{-{\frac{1}{1+H}}}\log \max_{\vert x\vert\le R} u(t,x)
\ge\lambda\hskip.2in \text{a.s.}
$$
Finally, recall that $\lambda$ satisfies condition \eqref{e.2.22}, and can thus be chosen arbitrarily close to $\hat{c}_{H,t}^{-1/(1+H)}$. Our proof is thus easily concluded. 
\end{proof}

\subsection{Proof of the upper bound}
The proof of the upper bound is based on a quantification of the fluctuations of $u$ in boxes around the points $x_{j}\in\cn_{R}$,  defined in the proof of Proposition \ref{prop:lower-bound}. We will first need an evaluation of the modulus of continuity of $u$ in the space variable.

\begin{proposition}  \label{prop9}
For any $\beta\in (0,2H-1/2)$, there  exists a constant  $C$ depending on $\alpha$, $H$ and $t$, such that for any $x,y \in \R$ and any  $p\ge 2$,
\begin{equation}\label{eq:bnd-Lp-increments-ut}
\EE \lc |u(t,x) - u(t,y) |^p\rp \le |x-y|^{p\beta} \exp \left( Cp^{ 1+ \frac 1{H}}\right).
\end{equation}
\end{proposition}

\begin{proof}
First we estimate the $L^2$ norm using  the Wiener chaos expansion of the solution and the notation used in the proof of Proposition  \ref{thm:exist-uniq-chaos}. In this way we can write
\begin{equation}\label{m0}
\EE \left ( |I_n(f_n(\cdot, t,x))- I_n(f_n(\cdot,t,y)) |^2 \right)
=n! \| f_n(\cdot, t,x)- f_n(\cdot, t,y) \|^2_{\mathcal{H}^{\otimes n}} 
= n! c_H^n   L_{n}(x,y),
\end{equation}
where we have set
\begin{equation*}
L_{n}(x,y)
=
\| \mathcal{F} f_n (\cdot, t,x)-\mathcal{F} f_{ n}(\cdot, t,y)\|^2_{  L^2([0,t]^n\times \R^n, \lambda^n \times \mu^{n})}.
\end{equation*}
Furthermore, it is readily checked that
\begin{equation}\label{m1}
L_{n}(x,y)
=
\frac 1{ (n!)^2}  
 \int_{[0,t]^n}  \int_{\R^n} K_{xy}(\xi)
\prod_{i=1}^n e^{-\frac{1}{2}(s_{\si(i+1)}-s_{\si(i)})|\xi_{\si(i)}+\cdots +\xi_{\si(1)} |^2} 
 |\xi_{\si(i)}|^{1-2H} d\xi  ds ,
\end{equation}
where
\begin{equation*}
K_{xy}(\xi) \equiv 
\left|  e^{-ix (\xi_{\sigma(n)}+ \cdots + \xi_{\sigma(1)})} 
-
e^{-iy (\xi_{\sigma(n)}+ \cdots + \xi_{\sigma(1)})}  \right|^2 
=
\left|  e^{-ix (\xi_{n}+ \cdots + \xi_{1})} 
-
e^{-iy (\xi_{n}+ \cdots + \xi_{1})}  \right|^2 .
\end{equation*}
Notice that one can recast identity \eqref{m1} as
\begin{equation*}
L_{n}(x,y)
=
\frac 1{n!}  
 \int_{S_{n}(t)}  \int_{\R^n} K_{xy}(\xi)
\prod_{i=1}^n e^{-\frac{1}{2}(s_{i+1}-s_{i})|\xi_{i}+\cdots +\xi_{1} |^2} 
 |\xi_{i}|^{1-2H} d\xi  ds,
\end{equation*}
where we recall that $S_{n}(t)$ is defined in Lemma \ref{lem:intg-simplex}. Furthermore, the kernel $K_{xy}$ can be bounded as follows, for all $\beta\in(0,1)$
\begin{equation*}
\lln K_{xy}(\xi) \rrn
\le
   |x-y|^{2\beta} \lln  \xi_{1}+\cdots+\xi_{n} \rrn^{2\beta}.
\end{equation*}
Making the change of variable $\xi_{i}+\cdots + \xi_{1}=\eta_{i}$, for all $i=1,2,\dots, n$, with the convention   $\eta_0=0$  and  using the bound  $|\eta_i-\eta_{i-1} |^{1-2H} \le |\eta_i |^{1-2H} + | \eta_{i-1} |^{1-2H} $ as in the proof of Proposition \ref{thm:exist-uniq-chaos}, we obtain 
\begin{equation*}
L_{n}(x,y)
\le
\frac {|x-y|^{2\beta}}{ n!}    \sum_{\al\in D_{n}}
 \int_{S_{n}(t)}  \int_{\R^n} 
\lp\prod_{i=1}^{n-1} e^{-\frac{1}{2}(s_{i+1}-s_{i}) \eta_{i}^2}  
|\eta_i|^{\alpha_i}\rp   
e^{-\frac{1}{2}(t-s_{n})\eta_{n}^2}  
|\eta_{n}|^{\alpha_n+2\beta}
d\eta ds,
\end{equation*}
where we recall that $D_{n}$ has been introduced in \eqref{b11}.
Integrating with respect to the variables $\eta_i$ and using  Lemma \ref{lem:intg-simplex} thus yields
\begin{equation*}
L_{n}(x,y)
\le c^n
\frac {|x-y|^{2\beta}}{ n!}    \sum_{\al\in D_{n}}
 \int_{S_{n}(t)} \lp\prod_{i=1}^{n-1} (s_{i+1}-s_{i})^{-\frac{\al_{i}+1}{2}} \rp 
 (t-s_{n})^{-\frac{\al_{n }+2\beta+1}{2}} \, ds,
\end{equation*}
where 
\[
c= \max\left\{ \int_{\R} e^{-\frac 12 x^2} |x|^{\delta}dx \,; \,  \delta\in [0,4]  \right\}.
\]
At this point we can repeat the discussion following inequality \eqref{h15}. We find that, in order to ensure the convergence of the integral above, we have to choose $\beta<2H-\frac12$. Then applying Lemma~\ref{lem:intg-simplex}, we end up with
\begin{equation*}
L_{n}(x,y)
\le
\frac {C^{n}|x-y|^{2\beta}}{ n!\Gamma(nH+\frac 12)}.
\end{equation*}
Plugging this relation into \eqref{m0} and applying the hypercontracticity property on a fixed chaos, we have thus obtained
\begin{equation*}
\| I_n(f_n(\cdot, t,x))- I_n(f_n(\cdot,t,y)) \|_{L^{p}(\Omega)}
\le
\frac {C^{n} p^{\frac{n}{2}}|x-y|^{\beta}} { \Gamma\lp \frac{nH}{2}+1\rp},
\end{equation*}
from which \eqref{eq:bnd-Lp-increments-ut} is obtained exactly as in Proposition \ref{thm:exist-uniq-chaos}.
\end{proof}

We are now ready to prove the upper bound part of Theorem \ref{thm:space-asymptotics}.

\begin{proposition}\label{prop:upper-bound}
Under the assumptions of Theorem \ref{thm:space-asymptotics}, for all $t>0$ we have
\begin{equation}\label{eq:spatial-lower-bound}
\liminf_{R\to\infty}(\log R)^{-{\frac{1}{1+H}}}\log\lp \max_{\vert x\vert\le R}u(t,x) \rp
\le \lc \hat{c}_{H,t}\rc^{-\frac{1}{1+H}} \  \hskip.2in \text{a.s.},
\end{equation}
where $\hat{c}_{H,t}$ is defined in Corollary \ref{c.2.8} and is related to \eqref{eq:as-limit-in space} by Remark \ref{rmk-cHt-qHt}.
\end{proposition}

\begin{proof}
We shall use the same kind of notation as in the proof of Proposition \ref{prop:lower-bound}, sometimes with a slightly different meaning (which should be clear from the context).
 Fix $R>0$ and divide the interval $[-R,R]$ into subintervals $I_j$ with the same length, for $j=1, \dots ,\cn_{R}$ (notice that $\cn_{R}$ is now a cardinal instead of being a set as in Proposition \ref{prop:lower-bound}), of length less than or equal to $\ell$, for some $\ell  >0 $ to be chosen later. Pick one point $x_j$ of each interval $I_j$. By convention, we assume that $I_{1}$ contains 0, and we choose $x_{1}=0$. For any $x\in I_j$ we can write
 \[
 u(t,x)  \le  u(t,x_j) + |u(t,x)-u(t,x_j)|,
 \]
and hence:
 \[
 \max_{|x|  \le R } u(t,x) \le \max _{j} u(t,x_j) + \max_j  \sup_{x,y \in I_j , |x-y| \le  \ell} |u(t,x) -u(t,y)|.
 \]
 Therefore a simple elaboration of \eqref{l5.5} yields
  \begin{equation}\label{n1}
\log  \max_{|x|  \le R } u(t,x) \le \log 2
+  \log \max _{j}  \max \left(  u(t,x_j) , \sup_{x,y \in I_j, |x-y| \le  \ell} |u(t,x) -u(t,y)| \right).
 \end{equation}
Consider now $\lambda,\delta>0$.  Choose $\cn_{R}$ and $\ell$ large enough, so that $\cup_{j\le \cn_{R}}I_{j}$ covers $[-R,R]$. Owing to the stationarity of $u$, we can write
\begin{eqnarray*}
\PP\left(   \log  \max_{|x|  \le R } u(t,x)    \ge ( \lambda +\delta)  (\log R)^{ \frac 1{1+H}} \right)
&\le&
\sum_{j=1}^{\cn_{R}}
\PP\left(   \log  \max_{x\in I_{j}} u(t,x)    \ge ( \lambda +\delta)  (\log R)^{ \frac 1{1+H}} \right) \\
&\le&
\cn_{R} \, \PP\left(   \log  \max_{x\in I_{1}} u(t,x)    \ge ( \lambda +\delta)  (\log R)^{ \frac 1{1+H}} \right).
\end{eqnarray*}
Hence, inequality \eqref{n1} enables to get, for $R$ large enough
\begin{equation}\label{rab1}
 \PP\left(   \log  \max_{|x|  \le R } u(t,x)    \ge ( \lambda +\delta)  (\log R)^{ \frac 1{1+H}} \right)
\le \cn_{R} \, \PP \left( \log u(t,0) \ge \lambda  (\log R)^{ \frac 1{1+H}} \right)
+ F_{R},
\end{equation}
where $F_{R}$ is a fluctuation term given by
\begin{equation}\label{n2}
F_{R}=
\cn_{R} \, \PP \left(  \log \sup \lcl |u(t,x) -u(t,y)| : \, x,y \in I_1, |x-y |\le \ell \rcl
\ge \lambda (\log R)^{ \frac 1{1+H}} \right).
\end{equation}
Furthermore, according to Corollary \ref{c.2.8}, for any $\rho>0$ arbitrarily small and $R$ large enough we have
\[
\PP \left\{ \log u(t,0) \ge \lambda  (\log R)^{ \frac 1{1+H}} \right\}
\le  R^{- (\hat{c}_{H,t} -\rho) \lambda ^{1+H}}.
\]

Let us now specify our parameters: we assume that $\rho,\la$ satisfy $(\hat{c}_{H,t} -\rho) \lambda^{1+H} >1$, that is $\lambda^{1+H}$ is arbitrarily close to $\hat{c}_{H,t}^{-1/(1+H)}$, as in condition \eqref{e.2.22}. Then we can choose
$\cn_{R}=R^{\eta}$, with $1<\eta <  (\hat{c}_{H,t}  -\rho) \lambda^{1+H} $, and $\ell=CR^{1-\eta}$ so that $\cup_{j\le \cn_{R}}I_{j}$ covers $[-R,R]$. We get
\begin{equation}  \label{rab2}
\cn_{R} \PP \left\{ \log u(t,0) \ge \lambda  (\log R)^{ \frac 1{1+H}} \right\}
\le  R^{- \nu},
\quad\text{with}\quad
\nu= (\hat{c}_{H,t} -\rho) \lambda ^{1+H} -\eta,
\end{equation}
an notice that $\nu>0$.

The fluctuation term $F_{R}$ defined by \eqref{n2} can be handled as follows:
by Chebychev's inequality, we have
\begin{equation}\label{n3}
F_{R}
\le
R^\eta e^{-p\lambda (\log R)^{\frac 1{1+H}}}   
\EE \left[\sup \lcl |u(t,x) -u(t,y)|^{p} :\, x,y \in I_1, |x-y |\le \ell \rcl \right].
\end{equation}
Consider now $0<\ga<\beta<2H-\frac{1}{2}$ and $p$ such that $\beta-\ga>p^{-1}$. According to Garsia's lemma \cite{Ga}, we have
\begin{equation*}
\EE\lc \|u(t,\cdot)\|_{\ga,I_{1}}^{p} \rc
\le
c_{\ga,p} \int_{I_{1}^{2}} \frac{ \be \lp |u(t,x)-u(t,y)|^{p} \rp}{|x-y|^{\ga p + 2}} \, dx dy,
\end{equation*}
where $\|f\|_{\ga,I_{1}}$ stands for the $\ga$-H\"older norm of $f$ on the interval $I_{1}$. Plugging the result of Proposition \ref{prop9} into this inequality, we obtain
\begin{equation*}
\EE\lc \|u(t,\cdot)\|_{\ga,I_{1}}^{p} \rc
\le
c_{\ga,\beta,p} |I_{1}|^{p(\beta-\ga)} e^{C p^{1+\frac{1}{H}}}
= c_{\ga,\beta,p} \, \ell^{p(\beta-\ga)} e^{C p^{1+\frac{1}{H}}},
\end{equation*}
and going back to \eqref{n3}, we end up with
\begin{equation*}
F_{R}
\le
c_{\ga,\beta,p} \, R^\eta  \, \ell^{p(\beta-\ga)} 
\exp\lp -p\lambda (\log R)^{\frac 1{1+H}} + C p^{1+\frac{1}{H}}\rp.
\end{equation*}
Recall that $\ell$ is of the form $CR^{-(\eta-1)}$, and we also take $p$  such that the dominant term in the exponential above is $p\lambda (\log R)^{\frac 1{1+H}}$. This is achieved for instance by taking 
\begin{equation*}
p=\lp\frac{\la}{2C }\rp^{H} \, (\log(R))^{\frac{H}{1+H}},
\end{equation*}
which implies $Cp^{1+\frac 1H}=\frac 12 p\lambda (\log R) ^{\frac 1 {1+H}}$. In this way we obtain
\begin{equation}\label{n4}
F_{R} \le c_{\ga,\beta, p} \, R^{\eta- p(\beta-\ga)(\eta-1)} \, \exp\lp -\frac 12 \, p\lambda (\log R)^{\frac 1{1+H}} \rp.
\end{equation}
With the values of the parameters we have considered so far, observe that we can pick $p$ such that $p(\beta-\ga)>\frac{\eta}{\eta-1}$. In this case we have $\eta- p(\beta-\ga)(\eta-1)=-\ka$ with $\ka>0$, and we can recast \eqref{n4} as
\begin{equation}\label{n5}
F_{R} \le c_{\ga,\beta,p} \, R^{-\ka}.
\end{equation}
Gathering our bounds \eqref{rab2} and \eqref{n5} into \eqref{rab1}, we thus have
\begin{equation}\label{n6}
\PP\left(   \log  \max_{|x|  \le R } u(t,x)    \ge ( \lambda +\delta)  (\log R)^{ \frac 1{1+H}} \right)
\le
R^{- \nu} + R^{-\ka}.
\end{equation}

We can now conclude in the following way: relation \eqref{n6} asserts that
 \[
\sum_{m\ge 1} \PP\left\{   \log  \max_{|x|  \le 2^m } u(t,x)    \ge ( \lambda +\delta)  (\log 2^m)^{ \frac 1{1+H}} \right\} <\infty.
\]
Hence Borel-Cantelli's lemma applies, and we deduce, almost surely
\[
\limsup_{m\rightarrow \infty} 
\, (\log 2^m)^{- \frac 1{1+H}}  
\log \lp \max_{|x|  \le 2^m } u(t,x)\rp    \ge \lambda +\delta
\ge (\hat{c}_{H,t} +\rho) ^{-\frac 1{1+H}} +\delta.
\]
Because $\delta>0$ and $\rho>0$ are arbitrary, we thus get that almost surely
\[
\limsup_{m\rightarrow \infty} (\log 2^m)^{- \frac 1{1+H}}  \log  \max_{|x|  \le 2^m } u(t,x)    \le \hat{c}_{H,t}  ^{-\frac 1{1+H}},
\]
which implies
\[
\limsup_{R\rightarrow \infty} (\log R)^{- \frac 1{1+H}}  \log  \max_{|x|  \le R } u(t,x)    \le \hat{c}_{H,t} ^{-\frac 1{1+H}}.
\]
This completes the proof of the upper bound.
\end{proof}

\section{Appendix}

\begin{lemma}  \label{lemA} 
For any $\alpha \in (0,1)$ and  we have
\[
\sup_{\beta\ge 1} \int_\R |\xi|^\alpha \left| \int_{[-\beta,\beta]}  e^{ix \xi -\frac {x^2}2} dx\right|^2 d\xi <\infty.
\]
\end{lemma}

\begin{proof}
Clearly the integral over $\{ |\xi | \ \le 1\}$ is uniformly bounded. So, it suffices to consider the integral over the unbounded domain  $\{ |\xi | > 1\}$. We have
\begin{eqnarray*}
 \int_{-\beta} ^\beta  e^{ix \xi -\frac {x^2}2} dx &=& e^{-\frac {\xi^2}2}  \int_{-\beta}^\beta e^{-\frac 12 (x-i\xi)^2 }dx=
 e^{-\frac {\xi^2}2}   \int_{-\beta-i \xi }^{\beta -i\xi}  e^{-\frac{ x^2}2 }dx \\
 &=& 
 e^{-\frac {\xi^2}2}  \left(   \int_{-\beta}^\beta e^{-\frac {x^2} 2}  dx  - \int_{\beta} ^{\beta-i\xi}e^{-\frac{ x^2}2 }dx
 -  \int_{-\beta-i\xi } ^{-\beta}e^{-\frac{ x^2}2 }dx \right).
 \end{eqnarray*}
 Then, it suffices to consider the term
 \begin{align*}
 &\int_{|\xi| >1}   |\xi|^\alpha   e^{-\xi^2}  \left| \int_{\beta} ^{\beta-i\xi}e^{-\frac{ x^2}2 }dx \right|^2 d\xi
 =\int_{|\xi| >1}   |\xi|^{\alpha+2}   e^{- \xi^2 } \left| \int_0^1e^{-\frac 12 (\beta -i t\xi)^2 }dt \right|^2 d\xi \\
 &\le e^{-\beta^2}  \int_{|\xi| >1}   |\xi|^{\alpha+2}     \left| \int_0^1e^{-\frac 12 (1-t^2)\xi^2 }dt \right|^2 d\xi 
 \le c_{\ep} e^{-\beta^2}  \int_{|\xi| >1}   |\xi|^{\alpha-2+\ep}     \left| \int_0^1 \frac{dt}{(1-t^{2})^{1-\ep}} \right|^2 d\xi
 \\
 &\le c_{\ep} \int_{|\xi| >1}   |\xi|^{\alpha-2+\ep}    d\xi <\infty,
  \end{align*}
where $\ep$ is an arbitrarily small positive constant.
\end{proof}

\end{document}